\documentclass[twoside,11pt]{article}

\usepackage[margin=1in]{geometry}
\usepackage[sort&compress, numbers]{natbib} 

\usepackage{amsfonts, amssymb, amsmath, amsthm}
\usepackage[mathscr, mathcal]{eucal}
\usepackage{mathrsfs}
\usepackage{graphics}
\usepackage{setspace}
\usepackage{tensor}
\usepackage{graphicx}
\usepackage{verbatim}
\usepackage{cases}
\usepackage{float}
\usepackage{setspace} 
\usepackage[font=small,format=plain,labelfont=bf,up,textfont=it,up]{caption}

\usepackage{mathtools}
\mathtoolsset{showonlyrefs}

\usepackage{booktabs}

\usepackage{color}
\definecolor{darkred}{RGB}{100,0,0}
\definecolor{darkgreen}{RGB}{0,100,0}
\definecolor{darkblue}{RGB}{0,0,150}
\definecolor{purple}{rgb}{0.4,.1,.9}
\usepackage{hyperref}
\hypersetup{colorlinks=true, linkcolor=darkred, citecolor=darkgreen, urlcolor=darkblue}
\usepackage{url}

\renewcommand{\L}{\mathfrak{L}}

\newcommand{\1}{{\bf 1}}

\newcommand{\Rt}{\boldsymbol r}

\newcommand{\Dt}{{\boldsymbol\epsilon}}
\newcommand{\wh}{\widehat}
\newcommand{\fh}{\hat f} 
\newcommand{\fhc}{\hat f^+} 
\newcommand{\Xt}{\widehat{\boldsymbol{x}}^\textsc{ab}} 
\newcommand{\Xh}{\widehat{\boldsymbol{x}}} 
\newcommand{\X}{\boldsymbol{x}} 

\newcommand{\Yt}{\boldsymbol{y}} 
\newcommand{\Y}{\boldsymbol{y}} 

\newcommand{\Zt}{\boldsymbol{z}} 
 
\newcommand{\U}{U}
\newcommand{\Uh}{\wh U}
\newcommand{\Ut}{\wt U}
\newcommand{\V}{V}

\newcommand{\Vt}{\wt V}

\newcommand{\wt}{\widetilde}

\def\longto{\longrightarrow}

\def\ball{B}

\newtheorem{theorem}{Theorem}[section]
\newtheorem{lemma}{Lemma}[section]
\newtheorem{proposition}{Proposition}[section]

\numberwithin{equation}{section}
\newtheorem{contribution}{Contribution}

\theoremstyle{remark}
\newtheorem{remark}{Remark}[section]

\newtheorem{assumption}{Assumption}

\newcommand{\thmref}[1]{Theorem~\ref{thm:#1}}
\newcommand{\prpref}[1]{Proposition~\ref{prp:#1}}

\newcommand{\lemref}[1]{Lemma~\ref{lem:#1}}
\newcommand{\secref}[1]{Section~\ref{sec:#1}}
\newcommand{\appref}[1]{Appendix~\ref{sec:#1}}

\newcommand{\aspref}[1]{Assumption~\ref{asp:#1}}

\DeclareMathOperator{\rank}{rank}
\DeclareMathOperator{\trace}{trace}

\DeclareMathOperator{\range}{range}

\def\cA{\mathcal{A}}
\def\cB{\mathcal{B}}
\def\cC{\mathcal{C}}

\def\cE{\mathcal{E}}

\def\cJ{\mathcal{J}}

\def\cN{\mathcal{N}}

\def\cR{\mathcal{R}}
\def\cS{\mathcal{S}}
\def\cT{\mathcal{T}}

\def\bbE{\mathbb{E}}

\def\bbG{\mathbb{G}}

\def\bbN{\mathbb{N}}

\def\bbR{\mathbb{R}}

\newcommand{\E}{\operatorname{\mathbb{E}}}
\renewcommand{\P}{\operatorname{\mathbb{P}}}

\newcommand{\Cov}{\operatorname{Cov}}

\def\weak{\rightharpoonup}
\def\longweak{\xrightharpoonup{\quad}}

\def\eps{\varepsilon}
\def\implies{\ \Rightarrow \ }
\def\iff{\ \Leftrightarrow \ }

\begin{document}

\title{Confidence Bands for the Gradient Lines of a Density Function}
\author{
Ery Arias-Castro\footnote{Department of Mathematics and Halıcıoğlu Data Science Institute, University of California, San Diego, California}
\and
Wanli Qiao\footnote{Department of Statistics, George Mason University, Fairfax, Virginia}
}

\date{\today}
\maketitle

\begin{abstract}
We consider the problem of estimating the gradient ascent line of a density originating at a given point. Going beyond mere consistency, we establish a weak convergence result for a plugin estimator based on a kernel density estimator of the density. We then leverage that result to construct a confidence region for the gradient ascent line, including by bootstrap. 

\end{abstract}

\section{Introduction} 
\label{sec:introduction}

Given a density $f$ on $\bbR^d$ and a point $x_0\in \bbR^d$, we are interested in the integral curve driven by the gradient of $f$ with initial condition $x_0 \in \bbR^d$,
\begin{align}\label{odemodel}
\X'(t) = \nabla f(\X(t)), \quad t\geq 0; \quad \X(0) = x_0.
\end{align}
This is the gradient ascent flow associated with $f$ originating from $x_0$, and a foundational result in the theory of ordinary differential equations guarantees that it is well defined, reaching a critical point at the limit $t\to \infty$, if $\nabla f$ is Lipschitz. (This condition is not necessary, but it is sufficient.)

Interest in $\X$ stems from applications to clustering as defined by \citet{fukunaga1975}, and others \cite{cheng2004estimating, li2007nonparametric, roberts1997parametric}, where clusters correspond to basins of attraction of modes. When the density is of Morse regularity, it is known that this indeed defines a partition of the support up to a set of zero measure. 

Methodology aligning or inspired by this definition of clustering is often found under the umbrella name of `mean shift', which comes from an approach already suggested in \cite{fukunaga1975} for estimating the gradient lines originating at data points. See the review paper \cite{carreira2015clustering}. Theory for the problem of estimating the density gradient ascent lines is developed in \cite{cheng1995, comaniciu2002mean, cheng2004estimating, arias2016estimation, carreira2000mode}, and more recently, in \cite{arias2025clustering}, where a number of methods are shown to be consistent, including the mean shift as defined by \citet{cheng1995}. Some convergence rates are established in \cite{arias2016estimation, arias2025clustering}.

\subsection{Contribution}
Given an iid~sample from $f$, denoted $X_1,\dots,X_n$, we estimate $f$ by the following kernel density estimator
\begin{align}\label{f hat}
\fh(x) = \frac{1}{nh^d} \sum_{i=1}^n K\left( \frac{x-X_i}{h}\right),
\end{align} 
where $K$ is the kernel function on $\bbR^d$ and $h>0$ is the bandwidth. 
With $\fh$ in hand, we simply estimate the gradient ascent flow defined in \eqref{odemodel} by plugin, meaning, by the gradient ascent flow given by $\fh$ originating from the same point $x_0$,
\begin{align}\label{odemodel hat}
\Xh'(t) = \nabla \fh(\Xh(t)), \quad t\geq 0;\quad \Xh(0) = x_0.
\end{align}

It is shown in \cite{arias2025clustering}, for example, that if $f$ is twice continuously differentiable with bounded second derivative, and of Morse regularity (i.e., the Hessian is non-singular at critical points), and $\hat f$ is accurate to 2nd order, and if the starting point $x_0$ is in the basin of attraction of a mode (which is the case with probability one if it is drawn from $f$), then $\Xh$ is uniformly consistent.

In the present paper, we go much beyond consistency.
\begin{contribution}\label{con:weak}
Under some conditions and after proper scaling, we show that $\Xh-\X$ converges weakly to the solution of a particular linear stochastic differential equation (SDE) on any finite time interval $[0,T]$.
\end{contribution}

We then leverage this weak convergence result to obtain confidence bands.

\begin{contribution}\label{con:band}
We derive confidence regions for $\X$ on any given finite time interval $[0,T]$ in two ways: one construction is by plugin, while   the other construction is by bootstrap.
\end{contribution}

\subsection{More related work}

The estimation of the clusters as basins of attraction of the density modes is directly addressed in \cite{chen2017statistical}, where an error bound (in Hausdorff distance) is derived.

Density modes play a central role in the definition of clustering proposed by \citet{fukunaga1975}, and have been linked to clustering (and mixture modeling) for decades. There is by now a long line of research work on the estimation of density modes, see, e.g., \citep{hartigan1985dip, silverman1981using, minnotte1997nonparametric, dumbgen2008multiscale, burman2009multivariate, genovese2016non, tsybakov1990recursive, mokkadem2003law, genovese2016non, chen2017density}, and the many other references that appear in our recent paper on the topic \cite{arias2022estimation}.

As we have recently argued \cite{arias2023unifying,arias2023moving}, the definition of clustering proposed by \citet{fukunaga1975} is intimately related to the one proposed by \citet{hartigan1975clustering}, where clusters corresponds to connected components of upper level sets. This common understanding of clustering is known as {\em modal clustering} \cite{menardi2016review, chacon2020modal}.
While the estimation of level sets has received a substantial amount of attention \citep{polonik1995measuring, tsybakov1997nonparametric,rigollet2009optimal,mason2009asymptotic,walther1997granulometric,rinaldo2010generalized,singh2009adaptive, qiao2020asymptotics}, the estimation of the whole {\em cluster tree} (the structure that gathers all the clusters at all levels) has also been extensively studied \citep{stuetzle2003estimating, stuetzle2010generalized, rinaldo2012stability, chaudhuri2014consistent, wang2019dbscan, eldridge2015beyond, sriperumbudur2012consistency, steinwart2015fully, steinwart2011adaptive}.

We contribute to a scarce, but growing line of work which establishes weak convergence results for estimators of geometric features of the density. 
Weak convergence (and related bootstrap consistency) results are derived for plugin estimators of density level sets in \cite{chen2017density, qiao2019nonparametric, mammen2013confidence} and for plugin estimators of density ridges in \cite{chen2015asymptotic, qiao2021asymptotic, qiao2025confidence}.
Even closer to our work here is that of \citet{koltchinskii2007integral}, who aim at estimating the integral line of vector field $v: \bbR^d \to \bbR^d$ starting at some given $x_0 \in \bbR^d$, i.e., the curve $\X$ formally defined as the solution to
\begin{align}
\X'(t) = v(\X(t)), \quad t\geq 0; \quad \X(0) = x_0.
\end{align}
While the vector field $v$ is unknown, some noisy observations are available of the form
\begin{align}
V_i = v(X_i) + \xi_i, \quad i = 1, \dots, n,
\end{align}
where the $X_i$ are iid, independent of the $\xi_i$, also assumed iid. The setting is quite similar to ours: While in \cite{koltchinskii2007integral}, the vector field is estimated by kernel smoothing (Naradaya--Watson estimator), in ours $v = \nabla f$ is estimated by kernel density estimation. Consequently, the proof technique underlying the weak convergence result in \thmref{weak} is very much parallel. \citet{koltchinskii2007integral} use that result to derive a confidence region by plugin, and we do the same, but going beyond that, we also derive a confidence region by (smooth) bootstrap. We also build standardized confidence regions, which is not done in \cite{koltchinskii2007integral}.

\subsection{Content}
In \secref{setting} describes the setting, including assumptions on the density as well as the kernel function and bandwidth defining the estimator \eqref{f hat}.
In \secref{weak}, we state and proved the main result of the paper, which is the weak convergence of $\Xh-\X$ when properly scaled (see \thmref{weak}).  
This result is then used to derive confidence regions in \secref{plugin} (plugin construction) and \secref{bootstrap} (bootstrap construction).
\secref{numerics} describes some proof-of-concept numerical experiments. We raise a number of points of discussion in \secref{discussion}. 
Some technical issues are considered in several appendices.

\section{Setting}\label{sec:setting}
We consider a density $f$ on $\bbR^d$ satisfying the following.
\begin{assumption}\label{asp:f}
The density $f$ is $m_0 \ge 3$ times differentiable with bounded derivatives up to $m_0$th order, and converges to zero at infinity.
\end{assumption}
This can be relaxed somewhat, but we choose to work with relatively simple conditions that are reasonably general, to expedite some of the technical arguments that follow.

With $\|\cdot\|$ denoting the Euclidean norm of a vector, or more generally, the Frobenius norm of a matrix or a tensor, define
\begin{align}\label{kappa}
\kappa_m = \sup_{x\in\bbR^d} \|\nabla^m f(x)\|, \ m = 1, \dots, m_0.
\end{align}
While $\nabla f(x)$ denotes the gradient of $f$ at $x$, $\nabla^2 f(x)$ denotes its Hessian, and $\nabla^3 f$ the 3-way tensor gathering its third derivatives, etc. We will let $\partial^{(a_1, \dots, a_d)} f$ denote the partial derivative of $f$ with respect to the 1st variable $a_1$ times, with respect to the 2nd variable $a_2$ times, etc.
By \aspref{f}, $\kappa_k < \infty$ for all $k = 0,\dots,m_0$.
We focus on the interesting case of dimension $d \ge 2$.

We work with a kernel function $K$ satisfying the following.
\begin{assumption}\label{asp:K}
The kernel function $K$ is $m_1 \ge 2$ continuously differentiable, kills moments of order $1, \dots, k_0$, is supported on the unit ball, and integrates to one, and is such that
\begin{align}
\label{VC}
\Big\{\nabla^l K((\cdot - x)/h) : x \in \bbR^d, h > 0\Big\} \text{ is a VC function class for $l = 0, \dots, m_1$.}
\end{align}
\end{assumption}
In particular, by \aspref{K},
\begin{align}\label{K prop}
\int_{\bbR^d} K(x) dx = 1, \qquad 
\int_{\bbR^d} K(x) x^a dx = 0, \ 1 \le |a| \le k_0,
\end{align}
where, for $x = (x_1, \dots, x_d) \in \bbR^d$ and $a = (a_1, \dots, a_d) \in \bbN^d$, $x^a = x_1^{a_1} \cdot \cdots \cdot x_d^{a_d}$, and also $|a| = a_1 + \dots + a_d$. (By convention, $\bbN$ includes zero.)
We note that the technical assumption \eqref{VC} is satisfied, for example, when $K$ is spherically symmetric of the form $K(x) = \textsc{k}(\|x\|^2)$ with $\textsc{k}$ being $m_1+1$ continuously differentiable; or when $K$ is separable of the form $K(x) = \prod_{j=1}^d \textsc{k}_j(x_j)$ with each $\textsc{k}_j$ being $m_1+1$ continuously differentiable.

Recall that we use the kernel function $K$ in our estimator \eqref{f hat} for the density $f$. The accuracy of that estimator is very well understood. Introduce 
\begin{align}\label{eta}
\hat\eta_m = \sup_{x\in\bbR^d} \|\nabla^m \fh(x) - \nabla^m f(x)\|,
\end{align}
which is well-defined for $m = 0, \dots, m_0 \wedge m_1$.

\begin{lemma}\label{lem:etarate}
Under \aspref{f} and \aspref{K}, there is a constant $C$ depending on $K$ and $\bar\kappa$ such that, if $\max\{\kappa_0, \dots, \kappa_{m_0}\} \le \bar\kappa$, almost surely as $n\to\infty$,
\begin{align}
\label{etarate}
\hat\eta_m 
\le C \bigg(h^{(m_0-m) \wedge (k_0+1)} + \sqrt{\frac{\log n}{nh^{d+2m}}}\bigg), \quad \text{for $m = 0, \dots, m_0 \wedge m_1$,} 
\end{align}
whenever $h = h_n$ is such that $n h^d \gg \log n$ and $\log(1/h) \gg \log\log n$.
\end{lemma}
This result hinges on the fact, under the same conditions, 
\begin{align}\label{etavar} 
\hat\eta^{\rm sd}_m = \sup_{x\in\bbR^d} \big\|\nabla^m \fh(x) - \E[\nabla^m \fh(x)]\big\| 
\le C \sqrt{\frac{\log n}{nh^{d+2m}}}, \quad \text{for $m = 0, \dots, m_0 \wedge m_1$.}
\end{align}
The bound \eqref{etavar} is directly taken from \cite{gine2002rates}, which traces a direct line to pioneering work by Talagrand. See also \cite[Lem 2, Lem 3]{arias2016estimation}, which are derived as immediate consequences of results appearing in \cite{mason2011general,mason2012proving}, where more references are provided. 
(We mention all this because the bound \eqref{etavar} will be used later on. The fact that $C$ depends only on $K$ and on $f$ only through an upper bound on $\kappa_0, \dots, \kappa_{m_0}$ plays a role in the proof of \thmref{weak seq}.)

\section{Weak convergence}
\label{sec:weak}

Let $\cC[0,T]$ denote the class of continuous functions defined on $[0,T]$ with values in $\bbR^d$. (The dependency on $d$ is thus left implicit.)

\begin{theorem}\label{thm:weak}
In the setting of \secref{setting}, consider $x_0$ such that $f(x_0) > 0$, and the gradient ascent flows defined in \eqref{odemodel} and \eqref{odemodel hat}. Assuming that $h = h_n \to 0$ in such a way that $nh^{d+4}/\log n\to \infty$ and $n h^{d+1+2[(m_0-1) \wedge (k_0+1)]} \to 0$, the sequence of stochastic process $\sqrt{nh^{d+1}}(\Xh-\X)$ converges weakly in the space $\cC[0,T]$ to the Gaussian process $\U$ defined by the following SDE
\begin{align}\label{SDEMain}
d\U(t)
= \nabla^2 f(\X(t))\U(t)dt + \bigg[f(\X(t)) \int_{\bbR}\int_{\bbR^d} \nabla K(z) \nabla K\big(\tau \nabla f(\X(t)) + z\big)^\top dz d\tau\bigg]^{1/2}dW(t),
\end{align}
with initial condition $\U(0)=0$. Above, $W$ is a standard Brownian motion in $\bbR^d$.
\end{theorem}
\begin{remark}
Because the SDE \eqref{SDEMain} is linear, the solution can expressed as follows
\begin{align}\label{U}
\U(t) = \int_0^t M(t) M(s)^{-1} B(s) dW(s),
\end{align}
where, if we set
\begin{align}\label{A}
A(t) =  \nabla^2 f(\X(t)), \qquad M'(t) = A(t) M(t), \quad \text{with $M(0) = I$},
\end{align}
and 
\begin{align}\label{B}
B(t) = \bigg[f(\X(t)) \int_{\bbR}\int_{\bbR^d} \nabla K(z) \nabla K\big(\tau \nabla f(\X(t)) + z\big)^\top dz d\tau\bigg]^{1/2}.
\end{align}
The fact that the integral in \eqref{B} defines a positive semidefinite matrix (so that taking the square root is well-defined) is not apparent, but it is a consequence of our proof arguments further down. See also \lemref{B psd}.
\end{remark}

\begin{remark}
If we insist on working with a nonnegative kernel, then necessarily $k_0=1$, which limits the rate of convergence as $h$ needs to satisfy $n h^{d+5} \to 0$. 
Otherwise, we may as well choose a kernel which kills $k_0 \ge m_0-2$, in which case $h$ is required to satisfy $n h^{d+2m_0-1}\to 0$, which can be much more mild if $m_0$ is large. 
\end{remark}

The remainder of the section is dedicated to establishing the theorem.
For a vector-valued function $g$ defined on $[0,T]$, we will sometimes use the notation
\[\|g\|_T\, = \sup_{0\leq t \leq T} \|g(t)\|.\]

\subsection{Preliminaries}
We start with a basic upper bound on the estimation error. By an appeal to a basic stability theory of ODEs, e.g., \cite[Sec 17.5]{hirsch2012differential}, the following is derived in \cite{arias2016estimation}.
\begin{lemma}\label{lem:XConsist}
Under \aspref{f} and \aspref{K}, 
\begin{align}
\|\Xh(t)-\X(t)\| 
\le \frac{\hat\eta_1}{\kappa_2}[e^{\kappa_2 t}-1], \quad \text{for all } t\geq 0.
\end{align}
In particular, 
\begin{align}\label{X bound}
\|\Xh-\X\|_T 
\le \frac{\hat\eta_1}{\kappa_2}[e^{\kappa_2 T}-1].
\end{align}
\end{lemma}
 
\begin{proposition}\label{prp:decompo}
Under \aspref{f} and \aspref{K}, we have the following decomposition:
\begin{align}\label{xntdecomp}
\Xh(t)-\X(t) = \Zt(t)+\Dt(t), \quad t\in[0,T],
\end{align}
where $\Zt(t)$ is uniquely defined by the differential equation
\begin{align}\label{Z1ode}
\frac{d\Zt(t)}{dt}=\nabla \fh(\X(t))-\nabla f(\X(t))+\nabla^2 f (\X(t))\Zt(t), \quad \Zt(0)=0,
\end{align}
and $\Dt(t)$ is a remainder term satisfying 
\begin{align}\label{decompo remainder}
\|\Dt\|_T \le T e^{\kappa_2T} \big(\hat\eta_2 \|\Xh-\X\|_T\, + \kappa_3 \|\Xh-\X\|_T^2\big).
\end{align}
\end{proposition}

\begin{proof}
By the fact that $\Zt$ satisfies the differential equation \eqref{Z1ode}, it also satisfies the corresponding integral equation
\begin{align}\label{ZIntegral}
\Zt(t)=\int_0^t[\nabla \fh(\X(s)))-\nabla f(\X(s))]ds+\int_0^t\nabla^2 f (\X(s))\Zt(s)ds.
\end{align}
Derive
\begin{align}\label{DecompY}
\Yt(t)
&:=\Xh(t)-\X(t) \\
&=\int_0^t[\nabla \fh(\Xh(s))-\nabla f(\X(s))]ds\\
&=\int_0^t(\nabla \fh-\nabla f)(\Xh(s))ds+\int_0^t[\nabla f(\Xh(s))-\nabla f(\X(s))]ds \\
&=\int_0^t(\nabla \fh-\nabla f)(\X(s))ds+\int_0^t\nabla^2 f (\X(s))  \Yt(s)ds+\Rt(t), \label{YIntegral}
\end{align}
where
\begin{align}
\Rt(t) 
&= \int_0^t\big[(\nabla \fh-\nabla f)(\Xh(s)) - (\nabla \fh-\nabla f)(\X(s))\big]ds \\ 
&\qquad +\int_0^t \big[\nabla f(\Xh(s))-\nabla f(\X(s)) - \nabla^2 f (\X(s))(\Xh(s)-\X(s))\big]ds.
\end{align}
By Taylor expansions for $\nabla \fh-\nabla f$ and $\nabla f$,
\begin{align}\label{Reminder}
\|\Rt\|_T \leq T\hat\eta_2 \|\Yt\|_T\, + T\kappa_3 \|\Yt\|_T^2.
\end{align}
By the fact that $\Dt(t) =\Yt(t)-\Zt(t)$, from \eqref{ZIntegral} and \eqref{YIntegral}, we have
\begin{align*}
\Dt(t)=\int_0^t\nabla^2 f (\X(s)) \Dt(s)ds+\Rt(t).
\end{align*}
Hence, by the triangle inequality and Jensen's, for all $0 \le t \le T$,
\begin{align}\label{DIneq}
\|\Dt(t)\|
\leq \|\Rt\|_T\,+\int_0^t\|\nabla^2 f (\X(s))\| \; \|\Dt(s)\|ds,
\end{align}
which makes it possible to apply Gr\"onwall's inequality (\lemref{Gronwall}) to get
\begin{align}
\|\Dt(t)\|&\leq \|\Rt\|_T\, \exp\bigg\{\int_0^t\|\nabla^2 f (\X(s))\|ds\bigg\} \\
&\leq \|\Rt\|_T\,  e^{\kappa_2t},
\end{align}
which allows us to conclude by way of \eqref{Reminder}.
\end{proof}

Throughout, we use the following variant of Gr\"onwall's inequality. 
\begin{lemma}\label{lem:Gronwall}
Let $I = [a, \infty)$ with $a\in\bbR$.  Suppose $g: I \mapsto \bbR$ and $\alpha: I \mapsto \bbR$ are nonnegative continuous functions and $\kappa$ is a nonnegative constant, satisfying the inequality
\begin{align*}
g(t)\leq \kappa+\int_a^tg(s)\alpha(s)ds, \quad t\in I.
\end{align*}
Then
\begin{align*}
g(t)\leq \kappa \exp{\bigg(\int_a^t\alpha(s)ds\bigg)}, \quad t\in I.
\end{align*}
\end{lemma}

Based on \eqref{eta} and our assumption on $h$, $\max\{\hat\eta_0, \hat\eta_1, \hat\eta_2\} \to 0$ a.s., and in particular, by \eqref{X bound},
\begin{align}\label{consistent T}
\|\Xh-\X\|_T \to 0\quad \text{a.s.}.
\end{align}
In view of \prpref{decompo}, it thus suffices to establish the weak convergence of $\sqrt{nh^{d+1}}\Zt$ (with $\Zt$ defined in \eqref{Z1ode}) to the Gaussian process solution to \eqref{SDEMain} on the time interval $[0,T]$. 
In what follows, $o$ and $O$ terms hide constant factors that may only depend on $d$, $f$, $K$, $x_0$, and $T$.
We use $X$ to denote a (generic) random variable with density $f$. 

Denote by $\cC_0[0,T]$ the set of all $\bbR^d$-valued continuous functions $v:[0,T] \to \bbR^d$ such that $v(0) = 0$. Define the mapping $\mathscr{U}: \cC_0[0,T] \mapsto \cC[0,T]$ that sends $v\in \cC_0[0,T]$ to the solution of the integral equation
\begin{align}\label{DiffEquDef}
u(t)=v(t)+\int_0^t \nabla^2 f (\X(s))u(s) ds, \quad u(0)=0.
\end{align}
Note that since the solution of \eqref{DiffEquDef} is unique, $\mathscr{U}$ is well-defined as a mapping.
Fix $v_1, v_2\in \cC_0[0,T]$ arbitrary, and let $u_1(t)=\mathscr{U}v_1(t)$ and $u_2(t)=\mathscr{U}v_2(t)$. By the fact that
\begin{equation}
u_j(t) = v_j(t) + \int_0^t \nabla^2 f (\X(s))u_j(s) ds, \quad j = 1,2,
\end{equation}
applying the triangle inequality and Jensen's, we get 
\begin{align*}
\|u_1(t)-u_2(t)\|\leq \|v_1-v_2\|_T\,+\int_0^t\|\nabla^2 f (\X(s))\|\|u_1(s)-u_2(s)\|ds,
\end{align*}
and applying Gr\"onwall's inequality (\lemref{Gronwall}), we further get
\begin{align*}
\|u_1(t)-u_2(t)\|&\leq \|v_1-v_2\|_T\, \exp\bigg\{\int_0^t\|\nabla^2 f (\X(u))\|du\bigg\}\\
&\leq \|v_1-v_2\|_T\, e^{\kappa_2 T},
\end{align*}
valid for all $0\leq t\leq T$. 
We can rewrite this as the fact that, for any $v_1, v_2\in \cC_0[0,T]$,
\begin{align*}
\|\mathscr{U}v_1-\mathscr{U}v_2\|_T \leq e^{\kappa_2 T} \|v_1-v_2\|_T\,,
\end{align*}
so that $\mathscr{U}$ is a Lipschitz mapping with respect to the supnorm. 

Recall \eqref{B} and denote by $\V$ the stochastic process satisfying the SDE
\begin{align}\label{EtaSDE}
d\V(t) = B(t) dW(t), \qquad \V(0)=0.
\end{align}
Then $\mathscr{U}\V$ satisfies \eqref{SDEMain} with value zero at zero.
Consider the following two sequences of processes
\begin{align}
\Vt(t) &:=\sqrt{nh^{d+1}}\int_0^t[\nabla \fh(\X(s))-\nabla f(\X(s))]ds,\label{Vt} \\ 
\Ut(t) &:=\sqrt{nh^{d+1}}\Zt(t), \quad \text{with $\Zt$ defined in \eqref{Z1ode}}, \label{Ut}
\end{align}
defined on $[0,T]$, and note that $\mathscr{U}\Vt = \Ut$. 
As we will show below, $\Vt$ converges weakly in the space $\cC[0,T]$ to $\V$ defined in \eqref{EtaSDE}. Since $\mathscr{U}$ is Lipschitz, this will imply via the Continuous Mapping Theorem that $\Ut$ converges weakly to $\U$, which is what we need to establish.

It thus remains to establish the weak convergence of $\Vt$ as a sequence of stochastic processes indexed by $[0,T]$. 
As is common practice (see, e.g., \cite[Th~18.14]{van1996weak}), we do that by (i) establishing the weak convergence of the finite dimensional marginals; and (ii) establishing asymptotic equicontinuity (aka asymptotic tightness).
Define
\[p_t(s):= \1_{[0,t]}(s) = \1\{0 \le s \le t\}.\]
Let $\L\equiv\L(T)$ be the set of all functions which are bounded and continuous almost everywhere on $\bbR$ with support in $[0, T]$. Obviously, $\L$ is a linear space and $p_t\in \L$ for all $t\in[0, T]$. For any $p\in \L$, define the transformations 
\begin{align}\label{whJp}
\wt\cJ(p):=\int_0^T p(s)  \nabla \fh(\X(s))ds, \qquad \cJ(p):=\int_0^T p(s) \nabla f(\X(s))ds.
\end{align}
Then $\Vt$ can be written as follows
\begin{align}\label{EtaExp}
\Vt(t)=\sqrt{nh^{d+1}}\Big[\wt\cJ(p_t)-\cJ(p_t)\Big].
\end{align}
This leads us to studying $\wt\cJ - \cJ$. 
We decompose $\wt\cJ(p)$ into
$$\wt\cJ(p)=\frac{1}{nh^{d+1}}\sum_{i=1}^n\cE_i(p),$$
with
$$\cE_i(p):=\int_{\bbR} p(s)\nabla K\Big(\frac{\X(s)-X_i}{h}\Big)ds, \quad i=1,\dots,n.$$ 
Define
\begin{align}\label{Eoper}
\cE(p)=\int_{\bbR} p(s)\nabla K\Big(\frac{\X(s)-X}{h}\Big)ds,
\end{align}
and note that, for any $p \in \L$, $\cE(p), \cE_1(p),\cdots,\cE_n(p)$ are iid.

\subsection{Weak convergence of the finite dimensional marginals}
We first examine the first and second moments of $\wt\cJ - \cJ$, and then establish asymptotic normality.

\paragraph{Mean}
We have
\[\bbE[\wt\cJ(p)] -\cJ(p)
=\int_{\bbR} p(s) \Big(\bbE[\nabla \fh (\X(s))] - \nabla f (\X(s))\Big) ds.\]
As is well-known and based on integration by parts and the properties of the kernel function stated in \aspref{K},
\begin{align}\label{EVHRate}
\sup_{x \in \bbR^d} \big\|\bbE[\nabla \fh (x)] - \nabla f(x)\big\| = O(h^{(m_0-1) \wedge (k_0+1)}),
\end{align}
yielding 
\[\bbE[\wt\cJ(p)] -\cJ(p) = O(h^{(m_0-1) \wedge (k_0+1)}) \|p\|_\infty.\]
In particular, under the assumption $n h^{d+1+2[(m_0-1) \wedge (k_0+1)]} \to 0$, we have
\begin{align} \label{J bias}
\sqrt{nh^{d+1}}\Big(\bbE[\wt\cJ(p)] -\cJ(p)\Big) = o(1) \|p\|_\infty,
\end{align}
implying, for any $t \in [0,T]$,
\begin{align}\label{chi mean}
\E[\Vt(t)] = \sqrt{nh^{d+1}}\Big(\bbE[\wt\cJ(p_t)] -\cJ(p_t)\Big) \to 0 = \E[\V(t)], \quad n \to \infty.
\end{align}

For the record, we have found that 
\begin{align}
\bbE[\cE(p)]
= h^{d+1} \bbE[\wt\cJ(p)] 
= h^{d+1} \big(\cJ(p) + o(1) \|p\|_\infty\big) 
= O(h^{d+1}) \|p\|_\infty. \label{E E(p) 3}
\end{align}

\paragraph{Variance} 
For any $p, q\in\L$, 
\begin{align}\label{CovOrder}
\Cov\Big(\wt\cJ(p), \wt\cJ(q)\Big)&=\frac{1}{n^2h^{2d+2}}\Cov\bigg(\sum_{i=1}^n \cE_i(p), \sum_{j=1}^n \cE_j(q) \bigg)\\
&=\frac{1}{nh^{2d+2}}\Cov\Big(\cE(p), \cE(q) \Big)\\
&=\frac{1}{nh^{2d+2}}\Big( \bbE[\cE(p)\cE(q)^\top] - \bbE\cE(p)\bbE\cE(q)^\top\Big).
\end{align}
On the one hand, from \eqref{E E(p) 3}, we immediately get that
\begin{align}\label{COvpart1}
\frac{1}{nh^{2d+2}} \bbE\cE(p) \bbE\cE(q)^\top = O\Big(\frac{1}{n}\Big) \|p\|_\infty \|q\|_\infty.
\end{align}
On the other hand, 
\begin{align}\label{DecomPart1}
\bbE\big[\cE(p)\cE(q)^\top\big] =\int_{\bbR}\int_{\bbR} p(s)q(u)\bbE\bigg[\nabla K\bigg(\frac{\X(s)-X}{h}\bigg) \nabla K\bigg(\frac{\X(u)-X}{h}\bigg)^{\!\!\top} \bigg] duds,
\end{align}
with
\begin{align*}
&\bbE\bigg[\nabla K\bigg(\frac{\X(s)-X}{h}\bigg) \nabla K\bigg(\frac{\X(u)-X}{h}\bigg)^{\!\!\top} \bigg]\\
&\qquad = \int_{\bbR^d}  f(y) \nabla K\bigg(\frac{\X(s)-y}{h}\bigg) \nabla K\bigg(\frac{\X(u)-y}{h}\bigg)^{\!\!\top}dy\\
&\qquad = h^d\int_{\bbR^d} f(\X(s)-hz) \nabla K(z) \nabla K\bigg(\frac{\X(u)-\X(s)+hz}{h}\bigg)^{\!\!\top} dz,
\end{align*}
so that
\begin{align} \label{bbE E E}
&\bbE\big[\cE(p)\cE(q)^\top\big] \\
&=h^d \int_{\bbR}\int_{\bbR} p(s)q(u)\bigg[\int_{\bbR^d} f(\X(s)-hz) \nabla K(z)\nabla K\bigg(\frac{\X(u)-\X(s)+hz}{h}\bigg)^{\!\!\top} dz\bigg] duds \\
&=h^{d+1}\int_0^T\!\!\! \int_{-s/h}^{(T-s)/h}\!\!\! p(s)q(s+\tau h)\bigg[\int_{\ball(0,1)}\!\!\!\!\!\!  f(\X(s)-hz) \nabla K(z)\nabla K\bigg(\frac{\X(s+\tau h)-\X(s)}{h}+z\bigg)^{\!\!\top}dz\bigg] d\tau ds.
\end{align}
The issue is the integral over $\tau$, which is taken over an interval whose length diverges as $h \to 0$. In fact, it can be reduced to an integral over a bounded set. 
Indeed, there is $\gamma > 0$ that depends on $f$, $x_0$, and $T$ such that, for all $0 \le s \le t \le T$, 
\begin{align}\label{gamma}
\left\|\frac{\X(t) - \X(s)}{t-s}\right\| \ge \gamma,
\end{align}
and this implies that
\[\left\|\frac{\X(s+\tau h)-\X(s)}{h}\right\| \ge |\tau| \gamma,\] 
so that, for any $z \in \ball(0,1)$, by the triangle inequality,
\[\left\|\frac{\X(s+\tau h)-\X(s)}{h}+z\right\|
\ge |\tau| \gamma - 1,\]
so that 
\begin{align}\label{tau bounded}
|\tau| > 2/\gamma \implies \frac{\X(s+\tau h)-\X(s)}{h}+z \notin\ball(0,1) \implies \nabla K\bigg(\frac{\X(s+\tau h)-\X(s)}{h}+z\bigg) = 0,
\end{align} 
by the fact that $K$ is supported in the unit ball. Thus the integral with respect to $\tau$ can be taken over the fixed (and finite) interval $[-2/\gamma, 2/\gamma]$. We are thus dealing with a triple integral over a bounded domain, and it is straightforward to see that the integrand satisfies
\begin{align*}
&p(s)q(s+\tau h)  f(\X(s)-hz) \nabla K(z)\nabla K\bigg(\frac{\X(s+\tau h)-\X(s)}{h}+z\bigg)^{\!\!\top} \\
&\to p(s)q(s) f(\X(s)) \nabla K(z)\nabla K(\tau \X'(s)+z)^{\!\!\top}, \quad \text{as } h \to 0,
\end{align*}
for all almost all $z, \tau, s$. 
Applying Lebesgue dominated convergence theorem, we thus obtain
\begin{align}\label{COvpart2}
\frac{1}{h^{d+1}}\bbE\big[\cE(p)\cE(q)^\top\big]
\to \int_{\bbR} p(s)q(s) B(s)^2 ds.
\end{align}
With \eqref{CovOrder}, \eqref{COvpart1}, and \eqref{COvpart2}, we conclude that, for any given $p, q \in \L$,
\begin{align*}
n h^{d+1} \Cov\Big(\wt\cJ(p), \wt\cJ(q)\Big)
&\to \int_{\bbR} p(s)q(s) B(s)^2 ds.
\end{align*}
In particular, for any $t_1, t_2 \in [0,T]$, 
\begin{align}\label{cov asymp}
&\Cov\big(\Vt(t_1), \Vt(t_2)\big) \\
&=\Cov\bigg(\sqrt{nh^{d+1}}(\wt\cJ(p_{t_1}) - \cJ(p_{t_1})), \sqrt{nh^{d+1}}(\wt\cJ(p_{t_2}) - \cJ(p_{t_2}))\bigg) \\
&\rightarrow \int_\bbR p_{t_1}(s) p_{t_2}(s) B(s)^2 ds 
= \int_0^{\min(t_1, t_2)} B(s)^2 ds 
= \Cov\big(\V(t_1), \V(t_2)\big), \label{chi cov}
\end{align}
where the last equality comes from the definition of $\V$ in \eqref{EtaSDE}.

\paragraph{Asymptotic normality}
In view of \eqref{chi mean} and \eqref{chi cov}, it only remains to establish the asymptotic normality. We start with showing that, for any $p \in \L$, $\sqrt{nh^{d+1}}\big(\wt\cJ(p) -\cJ(p)\big)$ is asymptotically normal.
We start with  
\[\sqrt{nh^{d+1}}\big(\wt\cJ(p)-\bbE[\wt\cJ(p)]\big) = \frac{1}{\sqrt{nh^{d+1}}}\sum_{i=1}^n \big(\cE_i(p)-\bbE\cE_i(p)\big),\] 
whose normal limit results from by an application of Lyapunov's central limit theorem. (The fact that $h$ depends on $n$ keeps us from applying the classical central limit theorem.) To verify Lyapunov's condition, we bound the fourth moment of $\cE(p)-\bbE\cE(p)$, starting with
\begin{align}
\bbE \|\cE(p)\|^4
&= \bbE\big[(\cE(p)^\top \cE(p))^2 \big]\\
&=\bbE\left[\left\{\int_{\bbR} \int_{\bbR} p(s)p(s_1)\nabla K\bigg(\frac{\X(s)-X}{h}\bigg)^{\!\!\top}  \nabla K\bigg(\frac{\X(s_1)-X}{h}\bigg)dsds_1\right\}^2\right] \\
&=\int_{\bbR^{d+4}} \nabla K\bigg(\frac{\X(s)-y}{h}\bigg)^{\!\!\top} \nabla K\bigg(\frac{\X(s_1)-y}{h}\bigg) \\
&\hspace{0.5in} \times \nabla K\bigg(\frac{\X(s_2)-y}{h}\bigg)^{\!\!\top} \nabla K\bigg(\frac{\X(s_3)-y}{h}\bigg) \\
&\hspace{1in} \times f(y) p(s)p(s_1)p(s_2)p(s_3)dydsds_1ds_2ds_3 \\
&=h^d\int_{\bbR^{d+4}} \nabla K(z)^\top \nabla K\bigg(z+\frac{\X(s_1)-\X(s)}{h}\bigg) \\
&\hspace{0.5in} \times \nabla K\bigg(z+\frac{\X(s_2)-\X(s)}{h}\bigg)^{\!\!\top} \nabla K\bigg(z+\frac{\X(s_3)-\X(s)}{h}\bigg) \\
&\hspace{1in} \times f(\X(s)-hz) p(s)p(s_1)p(s_2)p(s_3) dzdsds_1ds_2ds_3 \\
&=h^{d+3}\int_{\bbR^{d+4}}\nabla K(z)^\top \nabla K\bigg(z+\frac{\X(s+\tau_1h)-\X(s)}{h}\bigg)\\
&\hspace{0.5in}\times\nabla K\bigg(z+\frac{\X(s+\tau_2h)-\X(s)}{h}\bigg)^{\!\!\top}\nabla K\bigg(z+\frac{\X(s+\tau_3h)-\X(s)}{h}\bigg) \\
&\hspace{1in}\times f(\X(s)-hz) p(s)p(s+\tau_1h)p(s+\tau_2h)p(s+\tau_3h)dzdsd\tau_1d\tau_2d\tau_3. \label{bbE E 4}
\end{align}
As we saw earlier in \eqref{tau bounded}, the integral is in fact over the bounded set $\ball(0,1) \times [0,T] \times [-\gamma/2,\gamma/2]^3$, so that, by Jensen's inequality, and using \aspref{f} and \aspref{K},
\begin{align}\label{Raw4thM}
\bbE\|\cE(p)\|^4 = O(h^{d+3}) \|p\|_\infty^4.
\end{align}
We then use the fact that, by the triangle inequality, the convexity of the mapping of a real number to its fourth power, and Jensen's inequality,
\begin{align}
\bbE \|\cE(p)-\bbE\cE(p)\|^4
&\leq \bbE\big[\big(\|\cE(p)\| + \|\bbE\cE(p)\|\big)^4\big] \\
&\le 2^3(\bbE\|\cE(p)\|^4+\|\bbE\cE(p)\|^4) \\
&\leq 2^4\bbE\|\cE(p)\|^4, \label{Center4thM}
\end{align}
so that, 
\begin{align*}
\left(\frac{1}{\sqrt{nh^{d+1}}}\right)^4 \sum_{j=1}^n\bbE \|\cE_j(p)-\bbE\cE_j(p)\|^4
= \frac1{n h^{2d+2}} O(h^{d+3}) \|p\|_\infty^4 
\rightarrow 0.
\end{align*}
With Lyapunov's condition verified, we obtain the asymptotic normality of $\sqrt{nh^{d+1}}(\wt\cJ(p)-\bbE\wt\cJ(p))$, and because of \eqref{J bias}, $\sqrt{nh^{d+1}}(\wt\cJ(p)-\cJ(p))$ converges to the same normal limit. 

Now, for a multivariate marginal of arbitrary dimension $m \ge 1$, for any $p_1,\dots,p_m\in \L$ and $a_1,\dots,a_m\in \bbR$, $\sum_{i=1}^m a_ip_i\in\L$, and hence $\sqrt{nh^{d+1}}\sum_{i=1}^m a_i(\wt\cJ(p_i)-\cJ(p_i))$ is asymptotically normal. Hence, by the Cramér--Wold device, the joint distribution of $\sqrt{nh^{d+1}} (\wt\cJ(p_1),\cdots,\wt\cJ(p_m))$ is asymptotically normal. In particular, for any $m \ge 1$ and any $t_1, \dots, t_m \in [0,T]$, $(\Vt(p_{t_1}),\cdots,\Vt(p_{t_m}))$ is asymptotically normal.

\subsection{Asymptotic Equicontinuity}
From \eqref{J bias} and the fact that $\|p_t\|_\infty = 1$ for all $t$, we have
\begin{align}\label{BiasUniform}
\sup_{t\in[0,T]}\Big\|\sqrt{nh^{d+1}}\big(\bbE\wt\cJ(p_t)-\cJ(p_t)\big)\Big\|\rightarrow 0 \quad \textrm{as} \quad n\rightarrow\infty,
\end{align}
so that it suffices to establish the asymptotic equicontinuity of
\begin{align}\label{ZetaPro}
\hat\varsigma(t):=\sqrt{nh^{d+1}}\Big(\wt\cJ(p_t)-\bbE\wt\cJ(p_t)\Big).
\end{align}
We begin with bounding the fourth moment as follows
\begin{align}
&\bbE\big\|\wt\cJ(p)-\bbE\wt\cJ(p)\big\|^4\\
&\qquad = \frac{1}{n^4h^{4d+4}}\bbE\bigg\|\sum_{j=1}^n\Big(\cE_j(p)-\bbE\cE_j(p)\Big)\bigg\|^4\\
&\qquad = \frac{1}{n^4h^{4d+4}}\Big[n(n-1)\Big(\bbE\|\cE(p)-\bbE\cE(p)\|^2\Big)^2+n\bbE\|\cE(p)-\bbE\cE(p)\|^4\Big].
\end{align}
First,
\[\bbE\|\cE(p)-\bbE\cE(p)\|^2
\le \bbE\|\cE(p)\|^2 
= \trace \bbE\big[\cE(p)\cE(p)^\top\big],
\]
and by \eqref{bbE E E} and \eqref{tau bounded}, and the Cauchy--Schwarz inequality at the end, 
\begin{align}
&\trace \bbE \big[\cE(p)\cE(p)^\top\big] \\
&=h^{d+1} \int_0^T\!\!\! \int_{-2/\gamma}^{2/\gamma}\!\!\! p(s)p(s+\tau h)\bigg[\int_{\ball(0,1)}\!\!\!\!\!\! \nabla K(z)^\top \nabla K\bigg(\frac{\X(s+\tau h)-\X(s)}{h}+z\bigg) f(\X(s)-hz)dz\bigg] d\tau ds \\
&=O(h^{d+1}) \int_0^T \int_{-2/\gamma}^{2/\gamma}\!\!\! |p(s)p(s+\tau h)| d\tau ds \\
&=O(h^{d+1}) \int_{-2/\gamma}^{2/\gamma} \left(\int_0^T p(s)^2 ds\right)^{1/2} \left(\int_0^T p(s+\tau h)^2 ds\right)^{1/2} d\tau \\
&=O(h^{d+1}) \int_0^T p(s)^2 ds.
\end{align}
Second, as we already saw in \eqref{Center4thM},
\[\bbE\|\cE(p)-\bbE\cE(p)\|^4 
= O(1)\, \bbE\|\cE(p)\|^4,\]
and by \eqref{bbE E 4} and the fact that the integral is as described just below that display,  
\begin{align}
&\bbE \|\cE(p)\|^4 \\
&=h^{d+3} \int_{-2/\gamma}^{2/\gamma} \int_{-2/\gamma}^{2/\gamma} \int_{-2/\gamma}^{2/\gamma} \int_0^T \int_{\ball(0,1)} \nabla K(z)^\top \nabla K\bigg(z+\frac{\X(s+\tau_1h)-\X(s)}{h}\bigg)\\
&\hspace{0.5in}\times\nabla K\bigg(z+\frac{\X(s+\tau_2h)-\X(s)}{h}\bigg)^{\!\!\top}\nabla K\bigg(z+\frac{\X(s+\tau_3h)-\X(s)}{h}\bigg) \\
&\hspace{1in}\times f(\X(s)-hz) p(s)p(s+\tau_1h)p(s+\tau_2h)p(s+\tau_3h)dzdsd\tau_1d\tau_2d\tau_3 \\
&=O(h^{d+3}) \int_{-2/\gamma}^{2/\gamma} \int_{-2/\gamma}^{2/\gamma} \int_{-2/\gamma}^{2/\gamma} \int_0^T \big|p(s)p(s+\tau_1h)p(s+\tau_2h)p(s+\tau_3h)\big| dsd\tau_1d\tau_2d\tau_3 \\
&=O(h^{d+3}) \int_{-2/\gamma}^{2/\gamma} \int_{-2/\gamma}^{2/\gamma} \int_{-2/\gamma}^{2/\gamma} \left(\int_0^T p(s)^4 ds\right)^{1/4} \left(\int_0^T p(s+\tau_1 h)^4 ds\right)^{1/4} \\
&\hspace{1in} \times \left(\int_0^T p(s+\tau_2 h)^4 ds\right)^{1/4} \left(\int_0^T p(s+\tau_3 h)^4 ds\right)^{1/4} ds d\tau_1d\tau_2d\tau_3 \\
&=O(h^{d+3}) \int_0^T p(s)^4 ds.
\end{align}
We have thus found that 
\begin{align}\label{FourthMo}
\bbE\big\|\wt\cJ(p)-\bbE\wt\cJ(p)\big\|^4
= \frac{1}{n^4 h^{4d+4}}\Big[n^2 O(h^{2d+2}) \|p\|_2^4 + n O(h^{d+3}) \|p\|_4^4\Big].
\end{align}

Now, for $t_1, t_2 \in [0,T]$, given that
\[\|p_{t_1} - p_{t_2}\|_2^2 = \|p_{t_1} - p_{t_2}\|_4^4 = |t_1-t_2|,\] 
we have 
\begin{align}\label{FourthMoDiff}
\bbE\|\hat\varsigma(t_1)-\hat\varsigma(t_2)\|^4
&= n^2h^{2d+2}\bbE\|\wt\cJ(p_{t_1}-p_{t_2}) - \bbE \wt\cJ(p_{t_1}-p_{t_2})\|^4\\
&= \frac{O(1)}{n^2h^{2d+2}} \Big[n^2h^{2d+2} \|p_{t_1} - p_{t_2}\|_2^4 + nh^{d+3} \|p_{t_1} - p_{t_2}\|_4^4\Big]\\
&= O(1) \bigg[|t_1-t_2|^2+\frac{1}{nh^{d-1}}|t_1-t_2|\bigg].
\end{align}
The second term on the RHS complicates things. (Compare with Example 2.2.12 in~\cite{van1996weak}.) We do have 
\[\bbE\|\hat\varsigma(t_1)-\hat\varsigma(t_2)\|^4 
=O(1) |t_1-t_2|^2, \quad \text{for all $t_1, t_2$ such that $|t_1-t_2|\geq 1 / (nh^{d-1})$.}\]
Introduce the grid $\cA=\{k/(nh^{d-1}): k=1,\dots, \lfloor nh^{d-1} T\rfloor\}$.
Following the standard Kolmogorov-type chaining argument, we obtain that, for all $\epsilon>0$
\begin{align}\label{EquiCont1}
\lim_{\delta\rightarrow0}\limsup_{n\rightarrow\infty}\P\bigg\{\sup_{t_1,t_2\in \cA,|t_1-t_2|\leq\delta}\|\hat\varsigma(t_1)-\hat\varsigma(t_2)\|\geq\epsilon\bigg\}=0.
\end{align}
We handle pairs of $t_1, t_2$ such that $|t_1-t_2| \le 1 / (nh^{d-1})$ by discretizing the space $[0,T]$. For that, introduce a rounding map $\pi: [0,T] \mapsto \cA$ satisfying $|t-\pi(t)|\leq1/(nh^{d-1})$ for any $t\in [0,T]$.
We have, for any $p \in \L$,
\begin{align}\label{wtcJ}
\big\|\wt\cJ(p) - \bbE \wt\cJ(p)\big\| 
&\le \int_0^T |p(s)| \big\|\nabla \fh(\X(s)) - \bbE[\nabla \fh(\X(s))]\big\| ds \\
&\le \|p\|_1\ \hat\eta_1^{\rm sd},
\end{align}
which via \eqref{etavar} implies that, for any $t_1, t_2 \in [0,T]$, almost surely,
\begin{align*}
\|\hat\varsigma(t_1)-\hat\varsigma(t_2)\|
&= \sqrt{n h^{d+1}} \big\|\wt\cJ(p_{t_1}-p_{t_2}) - \bbE \wt\cJ(p_{t_1}-p_{t_2})\big\| \\
&= \sqrt{n h^{d+1}} \|p_{t_1}-p_{t_2}\|_1\, O\big(\log(n)/(nh^{d+2})\big)^{1/2} \\
&= |t_1-t_2|\, O\big(\log(n)/h\big)^{1/2},
\end{align*}
which, in particular, gives
\begin{align}
\sup_{t\in[0,T]} \|\hat\varsigma(t)-\hat\varsigma(\pi(t))\|
= \frac1{nh^{d-1}}\, O\left(\frac{\log n}{h}\right)^{1/2}
= O\left(\frac{\sqrt{\log n}}{n h^{d-1/2}}\right) = o(1), \quad \text{a.s.},\label{EquiCont2}
\end{align}
by our assumptions on $h$.
With \eqref{EquiCont1}, \eqref{EquiCont2}, and the fact that
\begin{align*}
\|\hat\varsigma(t_1)-\hat\varsigma(t_2)\| 
&\leq\|\hat\varsigma(t_1)-\hat\varsigma(\pi(t_1))\|+\|\hat\varsigma(\pi(t_1))-\hat\varsigma(\pi(t_2))\|+\|\hat\varsigma(t_2)-\hat\varsigma(\pi(t_2))\|,
\end{align*}
we obtain 
\begin{align*}
\lim_{\delta\rightarrow0}\limsup_{n\rightarrow\infty}\P\bigg\{\sup_{t_1,t_2\in [0,T],|t_1-t_2|\leq\delta}\|\hat\varsigma(t_1)-\hat\varsigma(t_2)\|\geq\epsilon\bigg\}=0,
\end{align*}
establishing the asymptotic equicontinuity of $\hat\varsigma$.

\section{Plugin confidence regions}
\label{sec:plugin}

\subsection{Pivotal construction}
\label{sec:plugin pivotal}

Assume that the conditions underlying \thmref{weak} hold.
The obvious approach to obtaining a confidence region for $\X$ (as an element of $\cC[0,T]$) is to use $\Xh$ as pivot via the supnorm. This requires the derivation, in fact, estimation of the distribution of $\|\Xh - \X\|_T$. Under the same conditions as \thmref{weak}, the Continuous Mapping Theorem gives
\begin{align} \label{weak supnorm}
\sqrt{n h^{d+1}} \|\Xh - \X\|_T \longto Q := \|\U\|_T, \quad n \to \infty, \text{ in distribution}.
\end{align}
In particular, for any $\alpha \in (0,1)$, if $q_{1-\alpha}$ denotes the $(1-\alpha)$-quantile of $Q$,
\begin{align}
\P\Big(\sqrt{n h^{d+1}} \|\Xh - \X\|_T \le q_{1-\alpha}\Big) \longto 1-\alpha.
\end{align}
In other words, if
\[\widetilde\cR_{1-\alpha} = \big\{\Y \in \cC[0,T] : \|\Xh - \Y\|_T \le q_{1-\alpha}\big\},\]
then 
\begin{align}
\P\big(\X \in \widetilde\cR_{1-\alpha}\big) \longto 1-\alpha.
\end{align}
The issue, of course, is that we do not know the distribution of $Q$ because the process $\U$ is defined based on the density $f$. 

\begin{remark}
Although not directly useful in our context, we note that 
\[\{Q < q\} = \{\|\U(t)\| < q,\, \forall t \le T\} = \{\tau_q > T\},\]
if $\tau_q$ is the first exit time of $\U$ from $\ball(0,q)$, i.e., $\tau_q = \inf\{t : \|\U(t)\| \ge q\}$, and that if $\psi(s,y) = \P(\tau_q > T \mid \U(s) = y)$, then $\psi$ satisfies the following {\em Kolmogorov backward equation}
\begin{align}\label{backward equation}
-\psi'(s, y) = y^\top \nabla^2 f(\X(s)) \nabla \psi(s, y) + \frac{1}{2} \trace \left[ B(s)^2 \, \nabla^2 \psi(s, y) \right],
\end{align}
for $0 \le s < T$, with terminal condition $\psi(T, y) = \1\{y \in \ball(0,q)\}$.
\end{remark}

The plugin approach to estimating the distribution of $Q$ in \eqref{weak supnorm} consists in `estimating' the limiting process $\U$ defined by \eqref{SDEMain}. The obvious way to do that is by plugin, replacing $f$ with $\fh$ and $\X$ by $\Xh$ in that SDE, thus `estimating' $\U$ by the solution, $\Uh$, of
\begin{align}\label{SDEMain hat}
d\Uh(t)
= \nabla^2 \fh(\Xh(t))\Uh(t)dt + \bigg[\fh(\Xh(t)) \int_{\bbR}\int_{\bbR^d} \nabla K(z) \nabla K\big(\tau \nabla \fh(\Xh(t)) + z\big)^\top dz d\tau\bigg]^{1/2}dW(t),
\end{align}
with initial condition $\Uh(0)=0$.
Here too, we can express $\Uh$ more directly as follows
\begin{align}\label{U hat}
\Uh(t) = \int_0^t \wh M(t) \wh M(s)^{-1} \wh B(s) dW(s),
\end{align}
where
\begin{align}\label{A hat}
\wh A(t) = \nabla^2 \fh(\Xh(t)), \qquad
\wh M'(t) = \nabla^2 \fh(\Xh(t)) \wh M(t), \quad \text{with $\wh M(0) = I$},
\end{align}
and 
\begin{align}\label{B hat}
\wh B(t) = \bigg[\fh(\Xh(t)) \int_{\bbR}\int_{\bbR^d} \nabla K(z) \nabla K\big(\tau \nabla \fh(\Xh(t)) + z\big)^\top dz d\tau\bigg]^{1/2}.
\end{align}

Because $\hat\eta_2 = o(1)$ under our conditions on $h$, $\nabla^2 \fh \to \nabla^2 f$ uniformly (a.s.), and because $\Xh \to \X$ uniformly on $[0,T]$ (remember \eqref{consistent T}), we have almost surely
\[\sup_{t \in [0,T]} \|\nabla^2 \fh(\Xh(t)) - \nabla^2 f(\X(t))\| \to 0.\]
Then, by standard stability arguments making use of the fact that $\nabla^2 f$ is Lipshitz (by \aspref{f}), we get that 
\begin{align}\label{M conv}
\|\wh M - M\|_T \longto 0, \quad \text{almost surely},
\end{align}
which then implies that
\begin{align}\label{M conv inv}
\|\wh M^{-1} - M^{-1}\|_T \longto 0, \quad \text{almost surely},
\end{align}
by the fact that $M$ is continuous with $M(t)$ invertible for all $t \ge 0$ \cite[Lem 3.11]{teschl2012ordinary}.

Because $\hat\eta_1 = o(1)$ under our conditions on $h$, $\nabla \fh \to \nabla f$ uniformly (a.s.), and because $\Xh \to \X$ uniformly on $[0,T]$ (remember \eqref{consistent T}), we have almost surely
\[\sup_{t \in [0,T]} \|\nabla \fh(\Xh(t)) - \nabla f(\X(t))\| \to 0.\] 
Recalling, as we saw in the proof of \thmref{weak}, that the double integral in \eqref{B} is in fact over $[-2/\gamma,2/\gamma] \times \ball(0,1)$, with $\gamma$ defined in \eqref{gamma}, one can argue in the same way that the double integral in \eqref{B hat} is in fact over $[-2/\hat\gamma,2/\hat\gamma] \times \ball(0,1)$, for some $\hat\gamma$ that satisfies \eqref{gamma}, but with $\Xh$ in place of $\X$, and that $\hat\gamma$ is a.s.~bounded from below away from~0. 
Therefore, by dominated convergence (and the continuity of $\nabla K$, which is part of \aspref{K}), 
\begin{align}\label{B conv}
\|\wh B - B\|_T \longto 0, \quad \text{almost surely}.
\end{align}

With \eqref{M conv}-\eqref{M conv inv} and \eqref{B conv}, we are able to conclude that
\begin{align}\label{U conv}
\Uh \longto \U, \quad \text{in distribution}.
\end{align}
Then, by the continuity mapping theorem,
\begin{align}\label{Q conv}
\wh Q = \|\Uh\|_T \longto \|\U\|_T = Q, \quad \text{in distribution},
\end{align}
and in particular, in $\hat q_{1-\alpha}$ denotes the $(1-\alpha)$-quantile of $\wh Q$, 
\begin{align}
\hat q_{1-\alpha} \longto q_{1-\alpha}, \quad \text{in probability}.
\end{align}
In affirming this, we are using the fact that $Q$ has a continuous distribution function which is strictly increasing over its support, which we establish in \lemref{Q quant}.
Continuing, by the delta method,
\begin{align}
\P\Big(\sqrt{n h^{d+1}} \|\Xh - \X\|_T \le \hat q_{1-\alpha}\Big) \longto 1-\alpha.
\end{align}
Therefore, if we define the confidence region
\begin{align}\label{conf region plugin}
\wh\cR_{1-\alpha}^{\rm plugin} = \Big\{\Y \in \cC[0,T] : \|\Xh - \Y\|_T \le \hat q_{1-\alpha}/\sqrt{n h^{d+1}}\Big\},
\end{align}
it satisfies
\begin{align}\label{conf region plugin level}
\P\big(\X \in \wh\cR_{1-\alpha}^{\rm plugin}\big) \longto 1-\alpha.
\end{align}
We encapsulate the end product of our derivations into the following statement.
\begin{proposition}
Under the conditions \thmref{weak}, $\wh\cR_{1-\alpha}^{\rm plugin}$ defined in \eqref{conf region plugin} is a confidence region for $\X$ with asymptotic level $1-\alpha$.
\end{proposition}

\subsection{Studentized pivotal constructions}
\label{sec:plugin student}

A Studentized construction would be attractive since $\|\Xh(t) - \X(t)\|$ tends to increase as $t$ increases. However, it is not nearly as straightforward as the simple pivotal construction just described in \secref{plugin pivotal}.

The stochastic process $\U$ given in \eqref{U} has mean zero and covariance matrix 
\begin{align}\label{Sigma}
\Sigma(t) = \Cov \U(t) = M(t) \bigg[\int_0^t M(s)^{-1} B(s)^2 M(s)^{-\top} ds \bigg] M(t)^\top.
\end{align}
The path towards Studentization that one would first try starts with considering the standardization $\Sigma(t)^{-1/2}(\Xh(t)-\X(t))$ and then the Studentization $\wh\Sigma(t)^{-1/2}(\Xh(t)-\X(t))$, estimating $\Sigma$ by plugin
\begin{align}\label{Sigma hat}
\wh\Sigma(t) = \Cov\big(\Uh(t) \mid \fh\big) = \wh M(t) \bigg[\int_0^t \wh M(s)^{-1} \wh B(s)^2 \wh M(s)^{-\top} ds \bigg] \wh M(t)^\top,
\end{align}
following \eqref{A hat} and \eqref{B hat}. 
This plugin estimate is indeed consistent, as by \eqref{M conv}, \eqref{M conv inv}, and \eqref{B conv}, and by the Continuous Mapping Theorem, in probability,
\begin{align}\label{Sigma conv}
\|\wh\Sigma-\Sigma\|_T \to 0.
\end{align}
However, already the standardization above comes with issues. The fact that $\Sigma(0) = 0$ is not the only one, as looking at the integrand, $M(t)$ is non-singular, but $B(t)$ is singular (\lemref{B psd}). 
Nonetheless, as we argue in \secref{Sigma singular} in the case of a spherically symmetric kernel function $K$, generically, $\Sigma(t)$ is nonsingular for any $t > 0$. But even when this is the case, the standardization remains problematic because of what happens in the neighborhood of $t = 0$, where the leading term in $\Sigma(t)$ is singular in that $\Sigma(t) \sim t B(0)^2$ as $t\to 0$. For more details, see \secref{Sigma t=0}. 
This leaves little hope for a simple plugin Studentization to work. 

Below, we propose a number of alternatives.  

\subsubsection{Regularized Studentization}
A vanilla regularization of the plugin covariance matrix estimator is viable and results in a Studentization of the form $(\wh\Sigma(t) + \sigma_0 I)^{-1/2}(\Xh(t) - \X(t))$ for some (user-specified) $\sigma_0 > 0$. Following similar arguments, by the Continuous Mapping Theorem, under the condition of \thmref{weak},
\begin{align} \label{weak supnorm student}
\sqrt{n h^{d+1}} (\wh\Sigma + \sigma_0 I)^{-1/2} (\Xh - \X) \longweak (\Sigma + \sigma_0 I)^{-1/2} U;
\end{align}
and, at the same time, by another application of the Continuous Mapping Theorem, using \eqref{U conv} and \eqref{Sigma conv}, 
\begin{align}
(\wh\Sigma + \sigma_0 I)^{-1/2} \Uh \longweak (\Sigma + \sigma_0 I)^{-1/2} \U.
\end{align}
Therefore, as $n \to \infty$,
\begin{align}
\sqrt{n h^{d+1}} \big\|(\wh\Sigma + \sigma_0 I)^{-1/2} (\Xh - \X)\big\|_T \stackrel{\text{approx}}{\sim} \big\|(\wh\Sigma + \sigma_0 I)^{-1/2} \Uh\big\|_T,
\end{align}
which can be used, in the same way we did before, to derive an asymptotically exact $(1-\alpha)$-confidence region for $\X$ on $[0,T]$.  

\subsubsection{Spherically weighted process}
\label{sec:spherically weighted}
As we stated earlier, generically, $\Sigma(t)$ is non-singular for all $t>0$. When this is the case, $\sigma(t) := \sqrt{\trace \Sigma(t)} > 0$ for all $t > 0$, although it remains the case that $\sigma(0) = 0$. Ignoring this for a moment, a different departure point is the standardization $\sigma(t)^{-1}(\Xh(t)-\X(t))$, resulting in  and then the Studentization 
\begin{align}
\label{spherical student}
\text{$\hat\sigma(t)^{-1}(\Xh(t)-\X(t))$, \quad based on the plugin estimator $\hat\sigma(t)^2 := \trace \wh\Sigma(t)$}.
\end{align}
This is instead motivated by the standardization of $\|\Xh(t)-\X(t)\|$, and it is indeed the case, based on \eqref{Sigma conv}, that
\[\|\hat\sigma - \sigma\|_T \to 0,\]
and in particular, for any $t > 0$, 
\[\sqrt{nh^{d+1}} \big\|\hat\sigma(t)^{-1}(\Xh(t)-\X(t))\big\| \weak \|\sigma(t)^{-1} U(t)\|, \qquad \text{with } \E\big[\|\sigma(t)^{-1} U(t)\|^2\big] = 1.\]
However, although there is no longer any issue with rank deficiency,   there is still an issue at $t = 0$ that results in $\sup_{t \in [0,T]} \|\sigma(t)^{-1} U(t)\| = 0$ a.s.. This is a well-known behavior of solutions to SDEs, which here boils down to the approximation $\sigma(t)^2 \sim \trace[t B(0)^2]$, so that $\sigma(t) \asymp \sqrt{t}$, implying $U(t) = B(0) W(t) + o_P(\sqrt{t})$, and it is a classical fact that, almost surely,
\begin{align}
\label{Wt unbounded}
\sup_{t \in (0,T]} \frac{\|W(t)\|}{\sqrt{t}} = \infty.
\end{align}
Therefore, the plugin Studentization \eqref{spherical student} cannot work as is.

However, it is also a classical fact that, almost surely,
\begin{align}
\label{Wt bounded}
\limsup_{t \to 0^+} \frac{\|W(t)\|}{\sqrt{t \log\log(1/t)}} = \sqrt{2},
\end{align}
implying that, for any deterministic weight function $\rho : [0,T] \to (0, \infty)$ that is continuous and such that $\rho(t) \gg \sqrt{t \log\log(1/t)}$ as $t \to 0$, $U/\rho$ is a continuous process on $[0,T]$ (with value $0$ at the origin), which in particular satisfies, almost surely,
\[\sup_{t \in [0,T]} \frac{\|U(t)\|}{\rho(t)} \asymp 1.\]
The idea then is to choose such a weight function and consider the following weighted process
\begin{align}
\label{spherical weighted}
\rho(t)^{-1}(\Xh(t)-\X(t)).
\end{align}
Since this is a continuous process on $[0,T]$ (with value $0$ at the origin) and $U/\rho$ is a continuous process on $[0,T]$, by the Continuous Mapping Theorem, as continuous processes on $[0,T]$, 
\begin{align}
\sqrt{n h^{d+1}} \rho^{-1} (\Xh-\X) \longweak \rho^{-1} U,
\end{align}
resulting in 
\begin{align}
\sqrt{n h^{d+1}} \big\|\rho^{-1} (\Xh-\X)\big\|_T \longweak \big\|\rho^{-1} U\big\|_T.
\end{align}
It is also the case that 
\begin{align}
\big\|\rho^{-1} \Uh \big\|_T \longweak \big\|\rho^{-1} U\big\|_T,
\end{align}
so that, as $n \to \infty$, 
\begin{align}
\sqrt{n h^{d+1}} \big\|\rho^{-1} (\Xh-\X)\big\|_T \stackrel{\text{approx}}{\sim} \big\|\rho^{-1} \Uh\big\|_T.
\end{align}
The resulting confidence region takes the form
\begin{align}\label{conf region spherical weighted}
\wh\cR_{1-\alpha} = \Big\{\Y \in \cC[0,T] : \|\Xh(t) - \Y(t)\| \le \rho(t) \hat q_{1-\alpha}/\sqrt{n h^{d+1}},\, \forall t \in [0,T]\Big\},
\end{align}
where $\hat q_\alpha$ denotes the $\alpha$-quantile of $\big\|\rho^{-1} \Uh\big\|_T$. The discussion leading to this definition establishes that this region has size $1-\alpha$ in the large sample limit.

\begin{remark}
Although random, it is possible and tempting to choose $\rho$ of the form $\rho(t) = \hat\sigma(t) r(t)$ with $r: [0,T] \to (0, \infty)$ a continuous and deterministic function such that $r(t) \gg \sqrt{\log\log(1/t)}$ as $t \to 0$. Even then, there is no `canonical' choice for $r$. In general, a choice of $\rho$ results in a particular thickness profile for the confidence band.
\end{remark}

\begin{theorem}
\label{thm:weighted_convergence}
Assume the conditions of \thmref{weak} hold. 
Let $\rho(t)$ be a continuously differentiable, positive, and bounded function on $(0, T]$, satisfying $\rho(t) \gg \sqrt{t \log \log (1/t)}$ and $|t \rho'(t) / \rho(t)| = O(1)$ as $t \downarrow 0$.  
Then, the sequence of stochastic process $\sqrt{nh^{d+1}}(\Xh-\X)/\rho$ converges weakly in the space $\cC[0,T]$ to the Gaussian process $\U/\rho$, where $\U$ is the same process as defined in \thmref{weak}.
\end{theorem}

The (lengthy) proof of this result is in \appref{weighted_convergence_proof}.

\section{Bootstrap confidence regions}
\label{sec:bootstrap}

We now consider an alternative approach to estimating the distribution of $Q$ based on a smoothed bootstrap. Such an approach is natural, but there is a catch: when using a higher-order kernel, it is possible that $\fh$ takes negative values, in which case it cannot play the role of sampling density in the bootstrap world. 
Instead, we correct it so that it is a proper density. There are multiple ways of doing that \cite{hall1993correcting}, and while any reasonable correction should work, we work with the one proposed in \cite{glad2003correction}, given by 
\begin{align}\label{fh boot}
\fhc(x) = \max\big\{0, \fh(x) - \hat c\big\},
\end{align}
where $\hat c \ge 0$ is the unique constant such that $\fhc$ integrates to one. 

However, we are immediately confronted with the difficulty that $\fhc$ may not even be differentiable, so that it is not even clear that it can be the base for a well-defined gradient ascent flow. And if we are able to overcome this difficulty, en route towards showing consistency, we would then have to deal with the fact that it does not satisfy \aspref{f}. Indeed, we would hope to prove that the bootstrap is consistent via what \citet[Sec 3.1.4]{shao1995jackknife} call the `imitation method' --- arguably the most direct way --- which would call for an application of \thmref{weak}.
Thankfully, it is not hard to show that \thmref{weak} applies under the weaker assumption of a density which is continuous on $\bbR^d$,  $m_0 \ge 3$ times continuously differentiable on $\{f > 0\}$, and, as before, converges to zero at infinity. The idea is to realize that, because the upper level sets are compact and a gradient ascent flow remains within an upper level set, that everything can be localized. In that case, after considering a fixed point $x_0$ with $t_0 = f(x_0) > 0$, the supremums in \eqref{kappa} and \eqref{eta} would be over the compact set 
\begin{align}
\label{S0}
\cS_0 = \{f \ge t_0/2\}.
\end{align}
In fact, everything happens in that set once $\{\fh \ge \hat t_0 = \fh(x_0)\} \subset \cS_0$, which happens eventually as $n\to\infty$ by the fact that $\hat\eta_0 \to 0$ under our assumptions. 
This line of reasoning is used in \cite{arias2016estimation}.

What makes the correction \eqref{fh boot} particularly compelling here is that, eventually as $n\to\infty$, $\fhc = \fh - \hat c$ on $\cS_0$. This comes from the obvious fact that $\hat c \to 0$, due to having $\hat\eta_0 \to 0$, and $\fhc(x) = \fh(x) - \hat c$ for any $x \in \{\fh \ge \hat t_0\}$. Thus, with probability one, when $n$ is large enough, not only is the gradient ascent flow of $\fhc$ originating at $x_0$ well-defined, it eventually coincides with that of $\fh$.  

We are now in a position to specify our proposal: It is to estimate the distribution of $\|\Xh - \X\|_T$ by the distribution of $\|\Xh^* - \Xh\|_T$ conditional on the sample $X_1, \dots, X_n$, where
\begin{align}\label{odemodel hat c}
\frac{d\Xh^*(t)}{dt} = \nabla \fh^*(\Xh^*(t)), \quad t\geq 0;\quad \Xh^*(0) = x_0,
\end{align}
and
\begin{align}\label{f hat c}
\fh^*(x) = \frac{1}{nh^d} \sum_{i=1}^n K\left( \frac{x-X^*_i}{h}\right),
\end{align}
with $X^*_1, \dots, X^*_n$ iid from $\fhc$.
Letting $\hat p_{1-\alpha}$ be the $(1-\alpha)$-quantile of $\|\Xh^* - \Xh\|_T$, define the confidence region
\begin{align}\label{conf region boot}
\wh\cR_{1-\alpha}^{\rm boot} = \Big\{\Y \in \cC[0,T] : \|\Xh - \Y\|_T \le \hat p_{1-\alpha}\Big\}.
\end{align}

\begin{remark}
We work with $\|\Xh^* - \Xh\|_T$ instead of $\sqrt{nh^{d+1}}\|\Xh^* - \Xh\|_T$ to emphasize the fact that one would not need to know the convergence rate in order to implement this approach.
\end{remark}

\begin{remark}
In practice, the distribution of $\|\Xh^* - \Xh\|_T$ would need to be approximated by Monte Carlo simulation, based on generating many samples of same size $n$ iid from $\fhc$. As is well-known, when the kernel is nonnegative so that $\fhc = \fh$, we generate $X^* = X_I + h Z$, where $I$ is uniform in $\{1, \dots, n\}$ and, independently, $Z$ is a random variable with density $h^{-d} K(\cdot/h)$. Generating such a $Z$ is particularly simple when the kernel $K$ is separable. In general, sampling from $\fhc$ is not straightforward. We could resort to rejection sampling with proposal density a density estimate obtained from a nonnegative kernel (with a proper choice of bandwidth) multiplied by~2, for example, to ensure domination, at least eventually. 
\end{remark}

Another necessary ingredient to make the imitation method work is to realize that \thmref{weak} applies (for the same given $x_0$) to a sequence of densities $f_n$ indexed by the sample size $n$ as long as their smoothness parameters
\begin{align}\label{kappa seq}
\kappa_{n,m} = \sup_{x\in\cS_0} \|\nabla^m f_n(x)\|, \quad m = 0, \dots, m_0,
\end{align}
remain bounded, and they approximate $f$ to 2nd order in that
\begin{align}\label{eta seq}
\eta_{n,m} = \sup_{x\in\cS_0} \|\nabla^m f_n(x) - \nabla^m f(x)\| \to 0, \quad \text{for all $m = 0, \dots, m_0$}.
\end{align}
This is indeed the case, formalized as the following.

\begin{theorem}\label{thm:weak seq}
Consider a density $f$ satisfying \aspref{f} and, after fixing $x_0$ such that $f(x_0) > 0$, consider the gradient ascent flow defined in \eqref{odemodel}. Let $f_n$ denote a sequence of densities on $\bbR^d$ which are $m_0$ continuous differentiable on $\cS_0$ such that, with definition \eqref{kappa seq}, $\sup_n \kappa_{n,m} < \infty$ for all $m \le m_0$, and such that \eqref{eta seq} holds. 
Choose a kernel satisfying \aspref{K} with $m_1 \ge m_0$, and let $\fh_n$ denote the kernel density estimator for $f_n$ with bandwidth $h$ based on an iid sample of size $n$.
Let $\X_n$ and $\Xh_n$ denote the gradient ascent flows of $f_n$ and $\fh_n$, respectively, originating at the same point $x_0$. 
Under the same conditions on the bandwidth $h$, the sequence of stochastic process $\sqrt{nh^{d+1}}(\Xh_n-\X_n)$ converges weakly in the space $\cC[0,T]$ to the same Gaussian process as in \thmref{weak}.
\end{theorem}

\begin{proof}[Proof of Theorem \ref{thm:weak seq}]
To establish the weak convergence of $\sqrt{nh^{d+1}}(\Xh_n - \X_n)$ to $\U$ using the so-called imitation method which consists in following the proof of \thmref{weak} and replacing the density $f$ (assumed fixed in that proof) with the sequence $f_n$. 

\textbf{Step 1: Uniform Trajectory Convergence} \\
The gradient ascent flow $\X_n(t)$ is driven by the vector field $\nabla f_n$, just as the empirical flow $\Xh(t)$ is driven by $\nabla \fh$. Because $f$ strictly increases along the gradient flow, the true trajectory $\X(t)$ is confined to $\{f \ge t_0\}\subsetneq \cS_0$. Thus, there exists a positive margin $\delta > 0$ between the compact trajectory $\X([0,T])$ and the boundary of the domain $\cS_0$. 

Because the vector fields converge uniformly ($\eta_{n,1} \to 0$), a standard ODE continuation argument ensures that for sufficiently large $n$, the deviation $\|\X_n(t) - \X(t)\|$ remains bounded by $\delta$ for all $t \in [0, T]$. This guarantees that the straight line segment connecting $\X_n(t)$ and $\X(t)$ lies entirely within $\cS_0$. Consequently, we can apply the Mean Value Theorem with the local Lipschitz constant $\sup_{x \in \cS_0} \|\nabla^2 f(x)\| \le \kappa_2$ to bound the spatial difference in the vector fields: $\|\nabla f(\X_n(t)) - \nabla f(\X(t))\| \le \kappa_2 \|\X_n(t) - \X(t)\|$.

Following the same ODE stability arguments used to establish Lemma \ref{lem:XConsist}, we can thus bound the trajectory difference by replacing the estimation error $\hat\eta_1$ with the uniform sequence error $\eta_{n,1}$, which yields
\begin{equation}
    \sup_{t \in [0,T]} \|\X_n(t) - \X(t)\| \le \frac{\eta_{n,1}}{\kappa_2}\big[e^{\kappa_2 T}-1\big] \to 0,
\end{equation}
where the limit follows from \eqref{eta seq}.

\textbf{Step 2: Linearization} \\
Following Proposition \ref{prp:decompo}, we can write
\begin{align}
    \Xh_n(t)-\X_n(t) = \Zt_n(t)+\Dt_n(t), \quad t\in[0,T],
\end{align}
where $\Zt_n(t)$ is uniquely defined by the differential equation
\begin{align}
    \frac{d\Zt_n(t)}{dt}=\nabla \fh_n(\X_n(t))-\nabla f_n(\X_n(t))+\nabla^2 f_n (\X_n(t))\Zt_n(t), \quad \Zt_n(0)=0.
\end{align}
Because of \eqref{kappa seq} and \eqref{eta seq}, we can obtain $\|\Dt_n\|_T = o_{\mathbb{P}}(\|\Xh_n-\X_n\|_T)$ following Proposition \ref{prp:decompo}. In view of Slutsky's Theorem, it suffices to establish the weak convergence of the linearized process $\sqrt{nh^{d+1}}\Zt_n$.

\textbf{Step 3: Lipschitz Mapping} \\
Define the sequence of integral mappings $\mathscr{U}_n: \cC_0[0,T] \mapsto \cC[0,T]$ corresponding to $f_n$ that sends $v\in \cC_0[0,T]$ to the solution of the integral equation
\begin{align}
u_n(t) = v(t) + \int_0^t \nabla^2 f_n(\X_n(s)) u_n(s) ds, \quad u_n(0)=0.
\end{align}
Let $\Vt_n(t) = \sqrt{nh^{d+1}}\int_0^t[\nabla \fh_n(\X_n(s))-\nabla f_n(\X_n(s))]ds$. By definition, $\sqrt{nh^{d+1}}\Zt_n = \mathscr{U}_n \Vt_n$. 
Let $u = \mathscr{U}v$, where $\mathscr{U}$ is given in \eqref{DiffEquDef}. We have
\begin{align}
\|u_n(t) - u(t)\| & = \Big\|\int_0^t \nabla^2 f_n(\X_n(s)) \big[u_n(s) - u(s)\big] ds + \int_0^t \Big[\nabla^2 f_n(\X_n(s)) - \nabla^2 f(\X(s))\Big] u(s) ds\Big\| \\
& \le \int_0^t \kappa_{n,2} \|u_n(s) - u(s)\| ds + t \delta_n \|u\|_T,
\end{align}
where $\delta_n = \|\nabla^2 f_n(\X_n) - \nabla^2 f(\X)\|_T$. By \eqref{eta seq} and the Lipschitz continuity of the Hessian, $\delta_n \le \eta_{n,2} + \kappa_3 \|\X_n - \X\|_T \to 0$. 
Applying Gr\"{o}nwall's inequality, we obtain:
\begin{align}
\|u_n - u\|_T \le T \delta_n \|u\|_T e^{\kappa_{n,2} T}.
\end{align}
Recall from the properties of $\mathscr{U}$ that $\|u\|_T \le \|v\|_T e^{\kappa_2 T}$. Substituting this into the bound reveals
\begin{align}
\|\mathscr{U}_n v - \mathscr{U} v\|_T \le \Big(T \delta_n e^{(\kappa_{n,2} + \kappa_2) T}\Big) \|v\|_T.
\end{align}
Because $\sup_n \kappa_{n,2} < \infty$ and $\delta_n \to 0$, the multiplier in the parentheses vanishes asymptotically. 
Consequently, we can write $\sqrt{nh^{d+1}}\Zt_n = \mathscr{U} \Vt_n + (\mathscr{U}_n \Vt_n - \mathscr{U} \Vt_n)$. As we establish below, $\Vt_n$ converges weakly and is therefore stochastically bounded (tight). Thus, the difference term $(\mathscr{U}_n \Vt_n - \mathscr{U} \Vt_n)$ is $o_{\mathbb{P}}(1)$. By the Continuous Mapping Theorem, $\mathscr{U} \Vt_n \weak \mathscr{U} \V = \U$, and Slutsky's Theorem guarantees that adding the $o_{\mathbb{P}}(1)$ difference preserves the same weak limit.

\textbf{Step 4: Finite Dimensional Marginals} \\
We evaluate the transformations $\wt\cJ_n(p_t)$ and $\cJ_n(p_t)$ defined analogously to \eqref{whJp} but evaluated along $\X_n(s)$. Specifically, we decompose $\wt\cJ_n(p) = \frac{1}{nh^{d+1}}\sum_{i=1}^n\cE_{n,i}(p)$, where 
\begin{align}
\cE_{n,i}(p) := \int_{\bbR} p(s)\nabla K\bigg(\frac{\X_n(s)-X_i}{h}\bigg)ds, \quad X_i \sim f_n,
\end{align}
and let $\cE_n(p)$ denote an independent copy of these terms. Note that $\Vt_n(t)=\sqrt{nh^{d+1}}[\wt\cJ_n(p_t)-\cJ_n(p_t)].$ In the following, we let $\bbE_n$ and $\Cov_n$ denote the expectation and covariance taken with respect to the density $f_n$.

For the mean, analogous to \eqref{EVHRate}, we have
\begin{equation}
\sup_{x \in \bbR^d} \big\|\bbE_n[\nabla \fh_n (x)] - \nabla f_n(x)\big\| = O(h^{(m_0-1) \wedge (k_0+1)}) \kappa_{n,m_0}.
\end{equation}
Thus, under the assumption $\sup_n \kappa_{n,m_0} < \infty$ and $n h^{d+1+2[(m_0-1) \wedge (k_0+1)]} \to 0$, for any $t \in [0,T]$,
\begin{align}
\bbE_n[\Vt_n(t)] = \sqrt{nh^{d+1}}\Big(\bbE_n[\wt\cJ_n(p_t)] -\cJ_n(p_t)\Big) \to 0 = \bbE[\V(t)], \quad \text{as } n \to \infty.
\end{align}

Similarly for the covariance and the fourth moment, because the densities $f_n$ are uniformly bounded over the domain of integration ($\sup_n \kappa_{n,0} < \infty$), we can verify $\Cov_n(\Vt_n(t_1), \Vt_n(t_2)) \to \Cov(\V(t_1), \V(t_2))$ for any $t_1, t_2 \in [0,T]$ (analogous to \eqref{cov asymp}), and $\bbE_n\|\cE_n(p)\|^4 = O(h^{d+3})\|p\|_\infty^4$ (analogous to \eqref{Raw4thM}). This uniform control yields the asymptotic normality of $(\Vt_n(t_1),\dots,\Vt_n(t_m))$ for any $m \ge 1$ and any $t_1, \dots, t_m \in [0,T]$ via the Cram\'{e}r--Wold device.

\textbf{Step 5: Asymptotic Equicontinuity of $\Vt_n$} \\
Let $\hat\varsigma_n(t) := \sqrt{nh^{d+1}}(\wt\cJ_n(p_t)-\bbE_n\wt\cJ_n(p_t))$. Notice that the fourth moment bound evaluated in \eqref{FourthMoDiff} depends only on integrals of the kernel gradient and the density supremum. Because $\sup_n \kappa_{n,0} < \infty$, we maintain uniformly
\begin{align}
\bbE_n\|\hat\varsigma_n(t_1)-\hat\varsigma_n(t_2)\|^4 = O(1) \bigg[|t_1-t_2|^2+\frac{1}{nh^{d-1}}|t_1-t_2|\bigg].
\end{align}
This uniform bound preserves the Kolmogorov chaining argument over the grid $\cA$, establishing \eqref{EquiCont1} for $\hat\varsigma_n$.

To handle the discretization error off the grid, in view of \eqref{wtcJ}, the bound requires controlling the stochastic fluctuation represented by
\begin{align}
\hat\eta_{n,1}^{\rm sd} := \sup_{x \in \cS_0} \|\nabla \fh_n(x) - \bbE_n[\nabla \fh_n(x)]\|.
\end{align}
Standard empirical process bounds for this fluctuation depend on the underlying density only through its supremum to control the variance, which can be bounded independent of $n$, because we have the uniform bound $\sup_n \sup_{x \in \cS_0} f_n(x) \le \sup_n \kappa_{n,0} < \infty$. Consequently, the standard asymptotic rate applies directly to the sequence $f_n$, yielding $\hat\eta_{n,1}^{\rm sd} = O(\sqrt{\log n / (n h^{d+2})})$ almost surely. Thus, the bounding step similar to \eqref{EquiCont2} holds:
\begin{align}
\sup_{t\in[0,T]} \|\hat\varsigma_n(t)-\hat\varsigma_n(\pi(t))\| = O\left(\frac{\sqrt{\log n}}{n h^{d-1/2}}\right) = o(1), \quad \text{a.s.}
\end{align}

Because the finite-dimensional marginals of $\Vt_n$ converge to those of $\V$, and the asymptotic equicontinuity of $\Vt_n$ holds, we conclude that $\Vt_n \weak \V$ in $\cC[0,T]$. Combined with the Continuous Mapping Theorem result in Step 3 and Slutsky's Theorem, this establishes $\sqrt{nh^{d+1}}(\Xh_n-\X_n) \weak \U$, completing the proof of Theorem \ref{thm:weak seq}.
\end{proof}

The sequence given by $\fhc = \fhc_n$ satisfies the conditions of \thmref{weak seq} almost surely, by the fact that $\fhc = \fh$ on $\cS_0$ eventually, and for  $m = 0, \dots, m_0$, $\hat\eta_m$ (corresponding to $\eta_{n,m}$) is such that $\hat\eta_m \to 0$, and $\wh\kappa_m$ (corresponding to $\kappa_{n,m}$) is such that $\wh\kappa_m \le \kappa_m + \hat\eta_m$, all by \eqref{etarate}. 
We are therefore able to apply \cite[Th 18.3.1]{lehmann2022testing} and get, via the {\em Continuous Mapping Theorem}, that
\begin{align}
\sqrt{nh^{d+1}}\|\Xh^* - \Xh\|_T \longto \|\U\|_T = Q, \quad \text{in distribution}.
\end{align}
Reasoning as we did in \secref{plugin pivotal}, we obtain the following corollary.
 
\begin{proposition}
$\wh\cR_{1-\alpha}^{\rm boot}$ defined in \eqref{conf region boot} is a confidence region for $\X$ with asymptotic level $1-\alpha$.
\end{proposition}

\begin{remark}
Bootstrap versions of the Studentized or weighted constructions described in \secref{plugin student}, and even in \secref{matrix weighted}, appear viable, but we did not look into the details. 
\end{remark}

\section{Numerical experiments}
\label{sec:numerics}

To evaluate the empirical performance of the proposed confidence regions, we conduct a simulation study. We compare the finite-sample coverage probabilities of the smoothed bootstrap and the SDE plug-in methods across two weighting schemes: raw and spherical weighting.

\subsection{Model setup and tuning parameters}

We simulate data from a two-dimensional Gaussian mixture model with two equally weighted components. The centers of the components are $\mu_1 = (-1.5, 0)^\top$ and $\mu_2 = (1.5, 0)^\top$, with a shared isotropic covariance matrix $\Sigma = 0.8 I_2$. 

The true gradient line of interest is initiated at $x_0 = (-0.3, 1)^\top$ and tracked over $t \in [0, T]$, where $T \in \{2, 5, 10, 25\}$. The ODE is approximated using Euler's method with a step size of $0.01$. A visualization of the density landscape, and the true trajectory, along with a superposition of $100$ empirical trajectories, is provided in Figure~\ref{fig:superposition}.

To assess asymptotic convergence, we evaluate the coverage across a grid of sample sizes: $n \in \{2000, 5000, 10000\}$. For each $n$, the bandwidth is set to $h = 0.6\, n^{-0.15}$ using the Gaussian kernel. This value is chosen to satisfy the theoretical assumptions regarding the relationship between $n$ and $h$ necessary for the asymptotic limits to hold in \thmref{weak}.

The results are aggregated over 5,000 independent Monte Carlo trials for each sample size. Within each trial, critical values are derived using $B=500$ replications of the respective smoothed bootstrap or SDE plug-in processes. The nominal target coverage probability is set to $1-\alpha = 0.90$. 

\begin{figure}[htbp]
    \centering
    \includegraphics[width=0.75\textwidth]{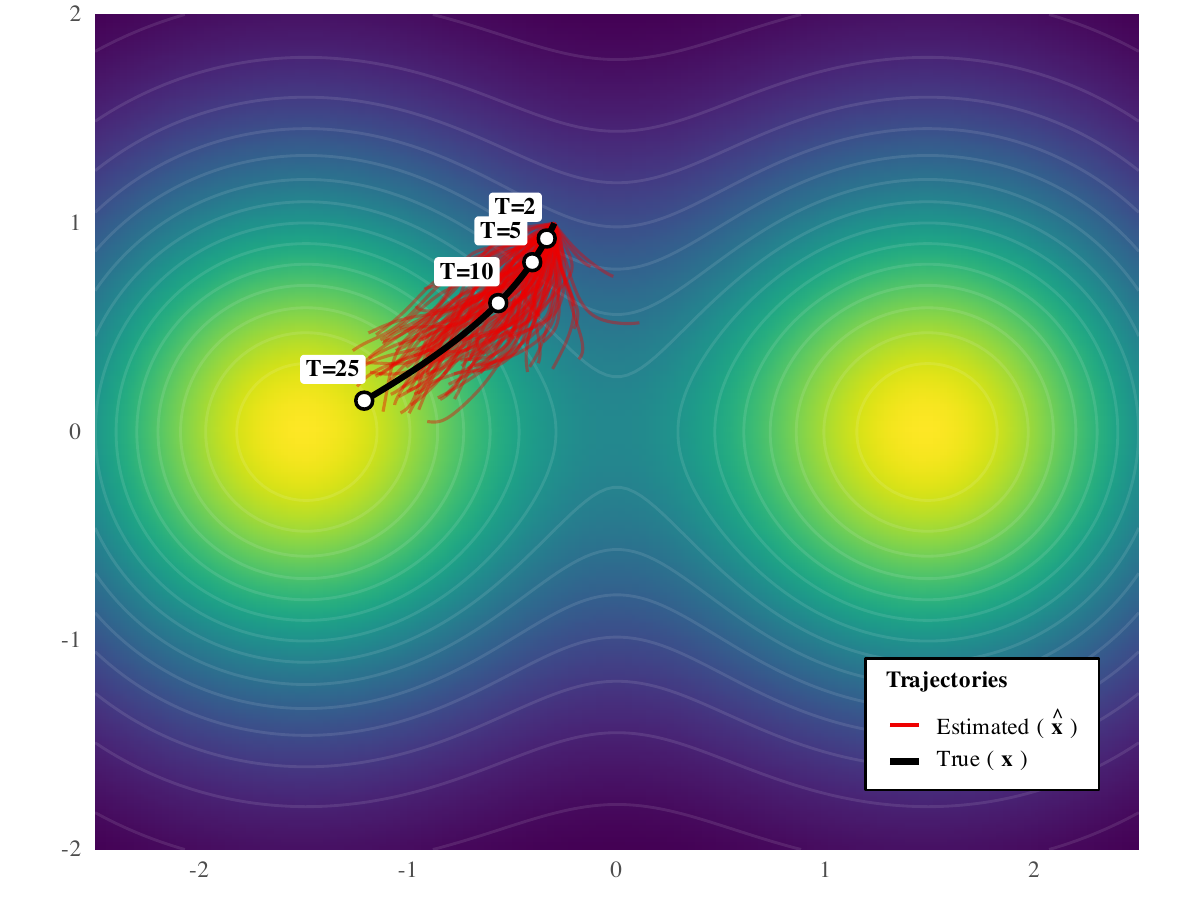}
    \caption{Superposition of $100$ estimated paths $\Xh$ (red) using sample size $n=10000$ over the true path $\X$ (black) on the true density landscape of a Gaussian mixture with two components. The trajectories originate at $x_0=(-0.3,1)$. }
    \label{fig:superposition}
\end{figure}

\subsection{Weighting schemes and quantile estimation methods}

To construct the confidence regions, we evaluate the supremum of the error process $\sqrt{nh^{d+1}}(\Xh(t) - \X(t))$ over $t \in (0, T]$ under two distinct weighting schemes:

\begin{enumerate}
    \item \textbf{Raw:} The raw error process is evaluated without any localized scaling. The critical value bounds the global supremum of the unweighted Euclidean error norm:
    \begin{equation}
        \|\sqrt{nh^{d+1}}(\Xh - \X)\|_T.
    \end{equation}
    
    \item \textbf{Spherical weighting:} The error process is scaled by a scalar sequence $\rho(t)$ meeting the requirements in \thmref{weighted_convergence}\footnote{While \thmref{weighted_convergence} assumes $\rho \in C^1(0, T]$, the differentiability is primarily utilized via the condition $|t \rho'(t) / \rho(t)| = O(1)$ as $t\to0$. Continuity suffices on $[\delta, T]$ for some $0<\delta<T$.}:
    \begin{equation}
        \left\| \sqrt{nh^{d+1}} \rho^{-1}(\Xh - \X) \right\|_T, \quad \text{where} \quad \rho(t) := \sqrt{t} \left[ \log \log \left( \frac{1}{\min(t, 0.3)} \right) \right].
    \end{equation}
\end{enumerate}

For each scheme, confidence bands are constructed using critical values derived from both the smoothed bootstrap and the SDE plug-in processes:

\begin{enumerate}
    \item \textbf{SDE plug-in:} we use the quantiles of $\big\|\Uh\big\|_T$ and $\big\|\rho^{-1} \Uh\big\|_T$ where $\Uh$ is the plug-in estimator of $U$. 
    
    \item \textbf{Smoothed bootstrap:} we use the quantiles of $\sqrt{nh^{d+1}}\|\Xh^* - \Xh\|_T$ and $\sqrt{nh^{d+1}}\|\rho^{-1} (\Xh^* - \Xh)\|_T$.
\end{enumerate}

\subsection{Pointwise asymptotics}

Before examining the simultaneous coverage probabilities, we verify the pointwise asymptotics using a sample size of $n=10000$. Figure~\ref{fig:pointwise_covariance} illustrates the evolution of the error and the corresponding covariance estimators at $T \in \{2, 5, 10, 25\}$. The empirical error distribution $\sqrt{nh^{d+1}}(\Xh - \X)$ is compared against its linearized counterpart $\sqrt{nh^{d+1}}\Zt$ as defined in \prpref{decompo}. To demonstrate the estimation variance across different sample paths, we evaluate two covariance matrix estimators corresponding to the above confidence region methods for 3 out of 100 Monte Carlo trials:
\begin{enumerate}
    \item $\widehat{\Sigma}_{\rm sde}$: The empirical covariance matrix of the $B=500$ simulated trajectories of the SDE plug-in process.
    \item $\widehat{\Sigma}_{\rm boot}$: The empirical covariance matrix of the $B=500$ simulated trajectories of the smoothed bootstrap process.
\end{enumerate}

As $T$ gets larger, the scaled error $\sqrt{nh^{d+1}}(\Xh(T) - \X(T))$ deviates more from its linearized version $\sqrt{nh^{d+1}}\Zt(T)$ in this finite sample experiment. Furthermore, the solid ellipses in Figure~\ref{fig:pointwise_covariance} demonstrate the considerable variation of these two covariance estimators across different random samples when $T$ gets larger.

\begin{figure}[htbp]
    \centering
    \includegraphics[width=\textwidth]{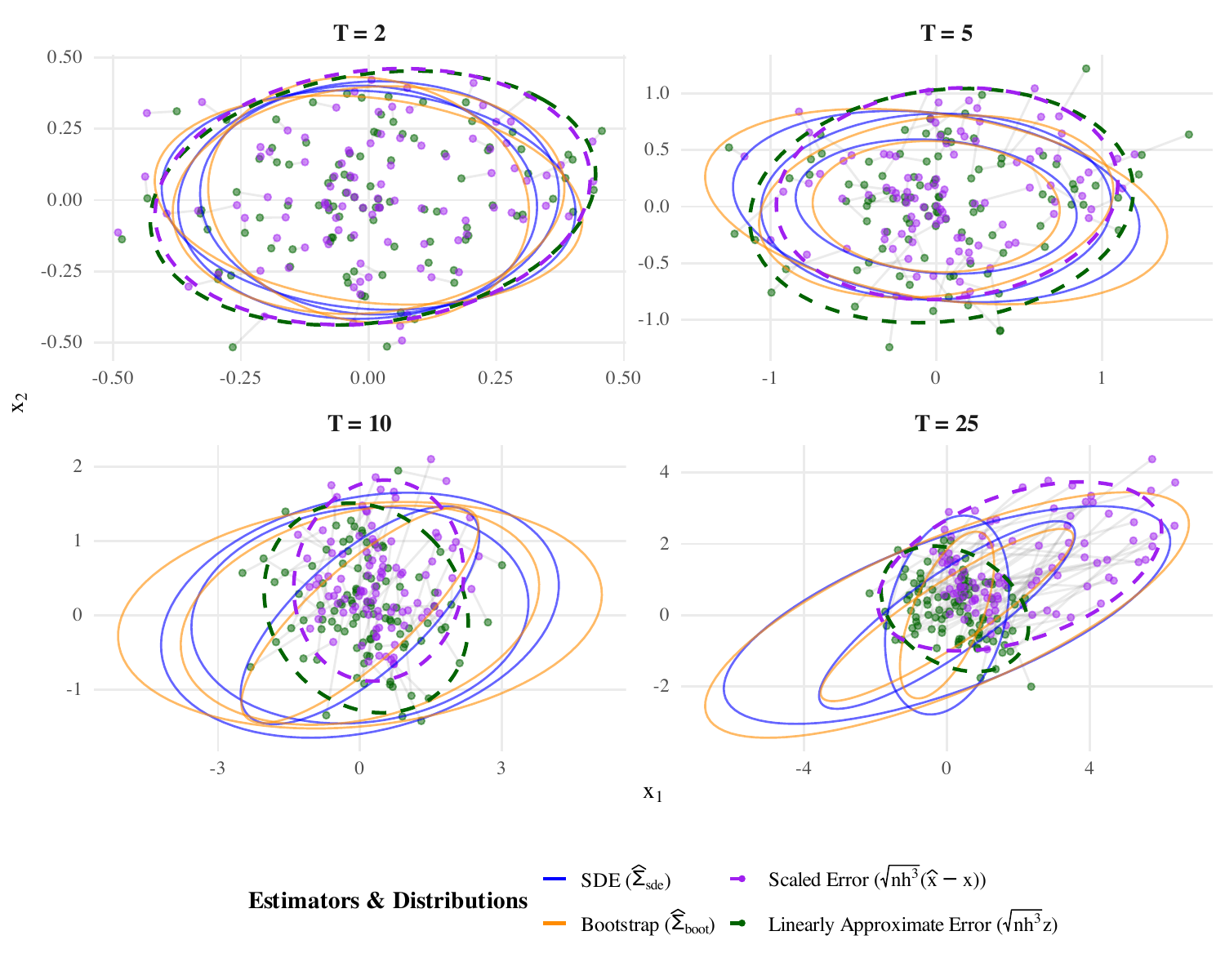}
    \caption{The evolution of the tracking error and covariance estimators across four time horizons $T \in \{2, 5, 10, 25\}$ using $n=10000$. The purple and green points are the scaled errors ($\sqrt{nh^3}(\Xh(T) - \X(T))$) and their linear approximates ($\sqrt{nh^3}\Zt(T)$, see \prpref{decompo}) for 100 random samples, where a pair of purple and green points are connected by a gray line if they are based on the same sample. The dashed ellipses of the same colors (purple and green) represent the $90\%$ highest density regions of Gaussian densities computed from the covariances of the scaled errors and their linear approximates, respectively. The solid ellipses display the $90\%$ highest density regions of Gaussian densities derived from the estimated covariance matrices computed via two distinct methods for 3 randomly selected samples: the SDE estimator $\widehat{\Sigma}_{\rm sde}$ (blue), and the smoothed bootstrap estimator $\widehat{\Sigma}_{\rm boot}$ (orange). This illustrates the estimation variance across random samples.} 
    \label{fig:pointwise_covariance}
\end{figure}

\subsection{Evaluation metrics}
The four confidence regions constructed are all in the form of $\cR_{1-\alpha}= \Big\{\Y \in \cC[0,T] : \Y(t) \in \cB_n(t), \forall t\in [0,T] \Big\}$, where $\cB_n$ depends on the specific construction methods. We report two distinct coverage probabilities evaluated at the target horizons $T \in \{2, 5, 10, 25\}$:

\begin{enumerate}
    \item \textbf{Time-aligned coverage:} By treating $\cR_{1-\alpha}$ as subsets of $\cC[0,T]$ and $\X$ as an element of $\cC[0,T]$, we evaluate the coverage of $\X$ by $\cR_{1-\alpha}$, that is, whether $\X(t) \in \cB_n(t)$, $\forall t\in [0,T]$.  Because our asymptotic theory is established in the sense of this coverage, the target coverage probability for this metric is $1-\alpha = 0.90$.
    \item \textbf{Geometric coverage:} By treating both $\X$ and $\cR_{1-\alpha}$ both as a subset of $\bbR^d$, we evaluate the coverage of $\{\X(t): t \in [0, T]\}$ by $\bigcup_{t\in [0, T]} \cB_n(t)$.  Since this coverage is implied by the time-aligned coverage, the target coverage probability for this metric is asymptotically bounded below by the nominal target, yielding $\ge 0.90$. See \secref{bands}.
\end{enumerate}

\subsection{Results}

The aggregated empirical coverage probabilities across different weighting schemes, quantile estimation methods, and sample sizes are summarized in Table~\ref{tab:time_aligned} (time-aligned coverage) and Table~\ref{tab:geo_tube} (geometric coverage), and visualized in Figure~\ref{fig:coverage_plots}. In Figure~\ref{fig:confidence_tubes} we also provide visualization of the shapes of the confidence regions generated by the two weighting schemes for a single representative trial at $n=10000$. 

As demonstrated in Figure~\ref{fig:pointwise_covariance}, as $T$ gets larger, the linearized theoretical error $\sqrt{nh^{d+1}}\Zt(T)$ becomes a progressively worse approximation of the actual finite-sample error $\sqrt{nh^{d+1}}(\Xh(T) - \X(T))$. Because our asymptotic confidence bands are developed based on the approximation of the latter using the former, achieving the target coverage probability at large $T$ necessitates a larger sample size $n$ (and correspondingly smaller $h$). This trend is clearly visible in Table~\ref{tab:time_aligned}: as $n$ increases from 2000 to 10000, the coverage probabilities for both the smoothed bootstrap and the SDE plug-in methods consistently approach the nominal $0.90$ level. Across the evaluated sample sizes, the SDE plug-in methods display over-coverage at early time horizons, while the smoothed bootstrap offers consistently tighter boundaries. 

\begin{figure}[htbp]
    \centering
    \includegraphics[width=\textwidth]{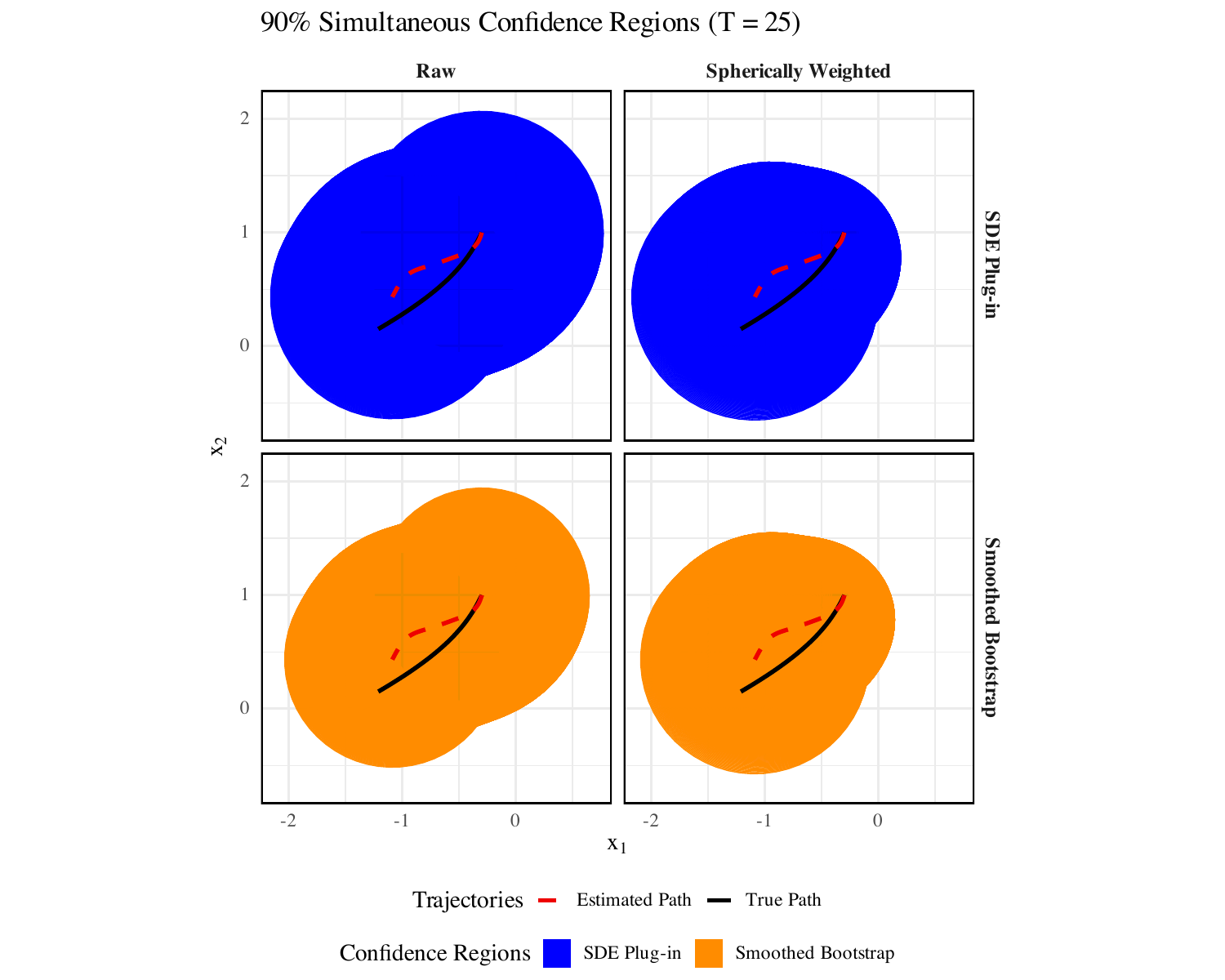}
    \caption{
    Asymptotic $90\%$ confidence regions for $\X(t), t \in [0, 25]$ using a random sample of size $n=10000$. The panels display the confidence regions constructed under the raw (left), and spherically weighted (right) schemes for both the SDE plug-in and smoothed bootstrap methods.}
    \label{fig:confidence_tubes}
\end{figure}

\begin{table}[htbp]
    \centering
    \caption{Time-aligned coverage probabilities (target $= 0.90$). }
    \label{tab:time_aligned}
    \begin{tabular}{ll cc cc}
        \hline\hline
        & & \multicolumn{2}{c}{\textbf{Smoothed Bootstrap}} & \multicolumn{2}{c}{\textbf{SDE Plug-in}} \\
        \cline{3-4} \cline{5-6}
        $n$ & $T$ & Raw & Spherical & Raw & Spherical \\
        \hline
        2000 & 2 & 0.898 & 0.898 & 0.936 & 0.952 \\
        & 5 & 0.846 & 0.855 & 0.922 & 0.948 \\
        & 10 & 0.789 & 0.812 & 0.894 & 0.924 \\
        & 25 & 0.698 & 0.758 & 0.782 & 0.834 \\
        \hline
        5000 & 2 & 0.897 & 0.897 & 0.938 & 0.956 \\
        & 5 & 0.860 & 0.869 & 0.926 & 0.950 \\
        & 10 & 0.805 & 0.830 & 0.892 & 0.925 \\
        & 25 & 0.761 & 0.807 & 0.807 & 0.851 \\
        \hline
        10000 & 2 & 0.896 & 0.896 & 0.937 & 0.956 \\
        & 5 & 0.864 & 0.872 & 0.927 & 0.951 \\
        & 10 & 0.824 & 0.842 & 0.896 & 0.925 \\
        & 25 & 0.822 & 0.860 & 0.849 & 0.879 \\
        \hline\hline
    \end{tabular}
\end{table}

\begin{table}[htbp]
    \centering
    \caption{Geometric coverage probabilities (target $\ge 0.90$).}
    \label{tab:geo_tube}
    \begin{tabular}{ll cc cc}
        \hline\hline
        & & \multicolumn{2}{c}{\textbf{Smoothed Bootstrap}} & \multicolumn{2}{c}{\textbf{SDE Plug-in}} \\
        \cline{3-4} \cline{5-6}
        $n$ & $T$ & Raw & Spherical & Raw & Spherical \\
        \hline
        2000 & 2 & 0.987 & 0.981 & 0.999 & 1.000 \\
        & 5 & 0.934 & 0.902 & 0.991 & 0.964 \\
        & 10 & 0.833 & 0.836 & 0.926 & 0.931 \\
        & 25 & 0.704 & 0.761 & 0.789 & 0.835 \\
        \hline
        5000 & 2 & 0.966 & 0.945 & 0.988 & 0.996 \\
        & 5 & 0.924 & 0.907 & 0.969 & 0.961 \\
        & 10 & 0.833 & 0.852 & 0.910 & 0.929 \\
        & 25 & 0.763 & 0.808 & 0.810 & 0.851 \\
        \hline
        10000 & 2 & 0.956 & 0.938 & 0.976 & 0.978 \\
        & 5 & 0.926 & 0.913 & 0.961 & 0.964 \\
        & 10 & 0.848 & 0.863 & 0.909 & 0.929 \\
        & 25 & 0.825 & 0.861 & 0.850 & 0.880 \\
        \hline\hline
    \end{tabular}
\end{table}

\begin{figure}[htbp]
    \centering
    \includegraphics[width=\textwidth]{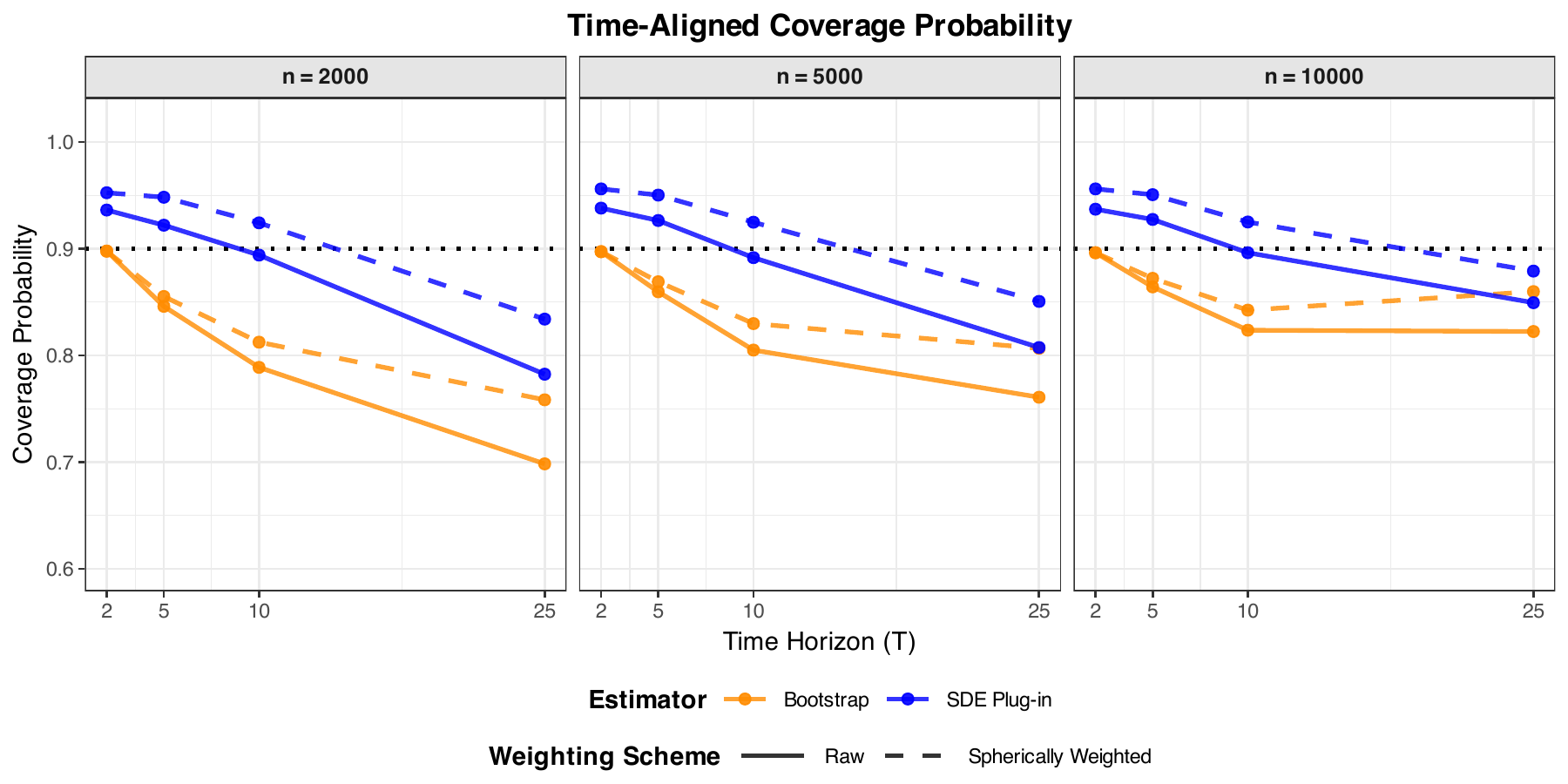}
    \vspace{0.5cm}
    
    \includegraphics[width=\textwidth]{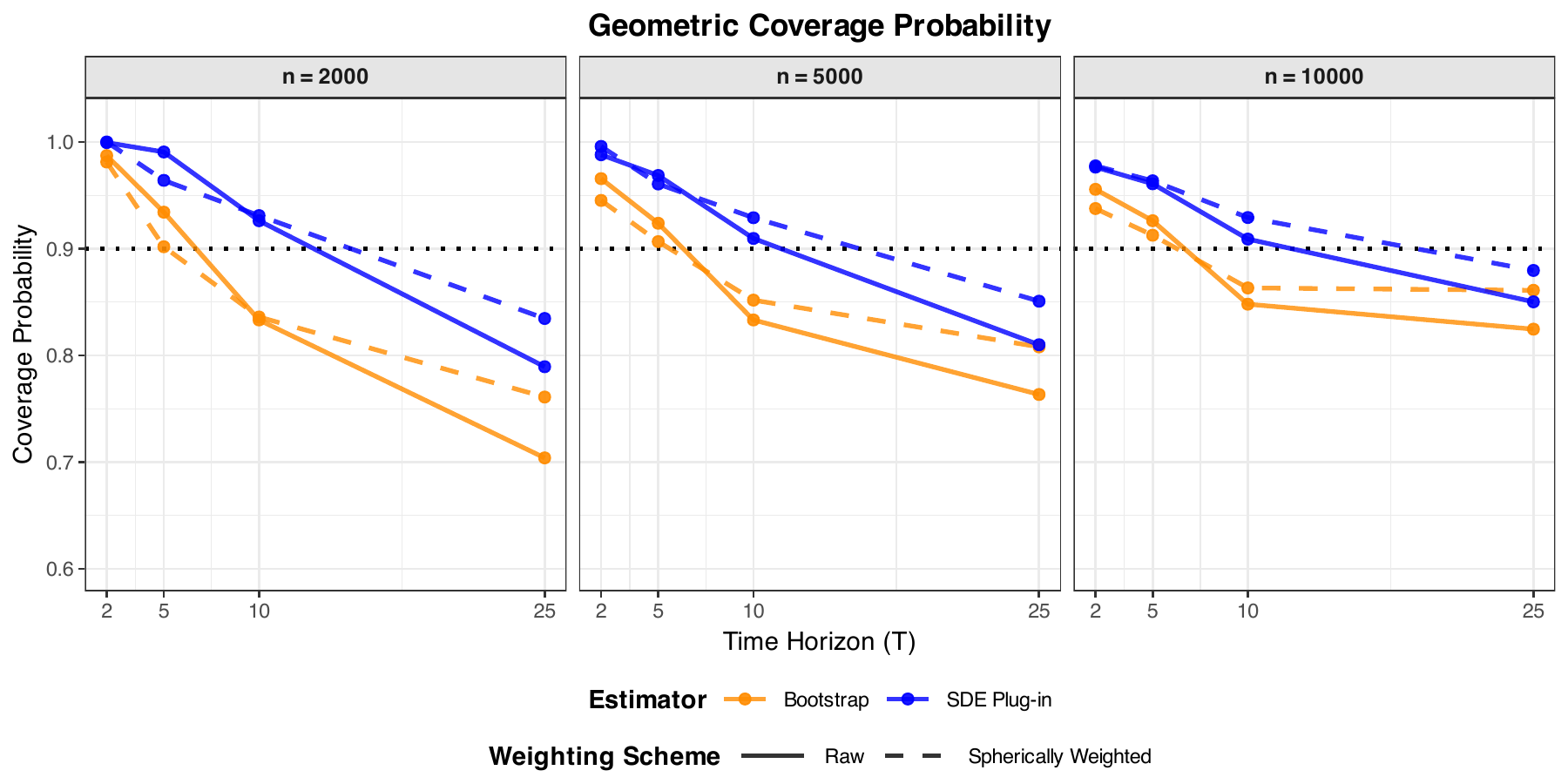}
    \caption{Coverage probabilities at time horizons $T \in \{2, 5, 10, 25\}$. \textbf{Top:} Time-aligned coverage probabilities. \textbf{Bottom:} Geometric coverage probabilities.}
    \label{fig:coverage_plots}
\end{figure}

\section{Discussion}
\label{sec:discussion}

Our work could be expanded in a number of ways and also leaves open a number of questions that arise naturally. We discuss some of these points below, in no particular order.  

\subsection{Different parameterizations of the gradient ascent flow}
Although the gradient ascent flow \eqref{odemodel} is arguably the most natural in the context of this paper, it is worth mentioning that this is not the flow originally discussed by \citet{fukunaga1975}. Indeed, the mean shift that they suggest leads to considering 
\begin{align}\label{odemodel mean shift}
\X_{\rm ms}'(t) = \frac1{f(\X_{\rm ms}(t))} \nabla f(\X_{\rm ms}(t)); \quad \X_{\rm ms}(0) = x_0.
\end{align}
This gradient ascent flow traces the same gradient ascent line, but at a different speed. Under our assumptions, this flow is also defined for all $t \ge 0$. 
In \cite{arias2023moving}, we consider another variant
\begin{align}\label{odemodel level}
\X_{\rm level}'(t) = \frac1{\|\nabla f(\X_{\rm level}(t))\|^2} \nabla f(\X_{\rm level}(t)); \quad \X_{\rm level}(0) = x_0,
\end{align}
which is also a reparameterization of the same gradient ascent flow, this time by the value of $f$, i.e., the level, in that $f(\X_{\rm level}(t)) = t + f(x_0)$. The flow is only defined for $0 \le t < f(x_*) - f(x_0)$ with $x_*$ being the critical point whose basin of attraction contains $x_0$. (In fact, it can be extended by continuity to $0 \le t \le f(x_*) - f(x_0)$.)

In general, consider the flow
\begin{align}\label{odemodel phi}
\Y'(s) = \phi(f, \Y(s)) \nabla f(\Y(s)); \quad \Y(0) = x_0.
\end{align}
The mean shift flow \eqref{odemodel mean shift} corresponds to $\phi(f, x) = 1/f(x)$, while the level flow \eqref{odemodel level} corresponds to $\phi(f, x) = 1/\|\nabla f(x)\|^2$. For the following discussion, it is enough if $\phi(f, \cdot)$ takes positive values and is Lipschitz in a neighborhood of $\X([0,T])$, where $\X$ is the original flow \eqref{odemodel}. In that case, the flow \eqref{odemodel phi} is well-defined for some time. In fact, it is a reparameterization of the original flow in that 
\begin{align}
\Y(s) = \X(\tau(s)), \quad \text{with} \quad \tau'(s) = \phi(f, \X(\tau(s))), \quad \tau(0) = 0,
\end{align}
defined for all $0 \le s \le S = \tau^{-1}(T)$.
Assume that $\phi(\fh, x) \to \phi(f, x)$ uniformly over the neighborhood just mentioned, and consider the plugin estimator flow
\begin{align}\label{odemodel phi hat}
\Yt'(s) = \phi(\fh, \Yt(s)) \nabla \fh(\Yt(s)); \quad \Yt(0) = x_0,
\end{align}
which can also be expressed as
\begin{align}
\Yt(s) = \Xh(\hat\tau(s)), \quad \text{with} \quad \hat\tau'(s) = \phi(\fh, \Xh(\hat\tau(s))), \quad \hat\tau(0) = 0,
\end{align}
eventually (as $n \to \infty$) defined for all $0 \le s \le S$.
In that case, under the same assumptions as in \thmref{weak}, using the conclusions of the theorem combined with the Continuous Mapping Theorem, in the space $\cC[0, S]$, 
\begin{align}
\sqrt{n h^{d+1}} (\Yt - \Y) \weak U \circ \tau,
\end{align}
where $U$ is the process defined in the same theorem. Recalling that $U$ is the solution to 
\[dU(t) = A(t) U(t) dt + B(t) dW(t), \quad U(0) = 0,\]
where $A$ is defined in \eqref{A} and $B$ is defined in \eqref{B}, 
$V = U \circ \tau$ is the solution to 
\[dV(s) = \phi(f, \tau(s)) A(\tau(s)) V(s) ds + \sqrt{\phi(f, \tau(s))} B(\tau(s)) dW(s), \quad V(0) = 0.\]

\subsection{Weak convergence on $[0,\infty)$?} 
\label{sec:infinite time}

Under \aspref{f}, in fact as long as $\kappa_1$ in \eqref{kappa} is finite, the gradient ascent flow \eqref{odemodel} is well-defined on $[0,\infty)$, and it is natural to want to derive the asymptotic behavior of $\Xh$ over the entire infinite time horizon $[0,\infty)$. It is known that $x_* = \lim_{t \to\infty} \X(t) = \X(\infty)$ exists and is a critical point of $f$. 

When $x_*$ is a non-degenerate mode, meaning that $H_* = \nabla^2 f(x_*)$ is negative definite, it is known that $\Xh$ is uniformly consistent for $\X$ over the infinite time horizon, 
\begin{align}\label{consistent}
\|\Xh-\X\|_\infty \to 0\quad \text{a.s.}.
\end{align}
This was shown in previous work \cite{arias2016estimation, arias2025clustering} under some assumptions similar to ours here. \aspref{f} and \aspref{K} are certainly sufficient. The issue that arises when deriving the limiting behavior of $\Xh-\X$ is that convergence at $t = \infty$ is substantially slower. While in \thmref{weak} we have shown that the convergence rate is $\sqrt{n h^{d+1}}$ in that $\sqrt{n h^{d+1}}(\Xh-\X) \weak U$ where $U$ is the Gaussian process solution to \eqref{SDEMain} starting at $U(0) = 0$, under the same conditions, at $t = \infty$ we have $\sqrt{n h^{d+2}}(\hat x_* - x_*) \weak Z$, where $Z$ is normal with zero mean and covariance matrix $f(x_*) H_*^{-1} C_0 H_*^{-1}$ with $C_0 := \int \nabla K(x) \nabla K(x)^\top dx$ \cite[Th 2.2(i)]{mokkadem2003law}.

The slower convergence rate at $t =\infty$ has consequences for the convergence at $t \to \infty$ sufficiently fast. Indeed, by linearizing $\nabla f$ near $x_*$, it can be shown \cite{barreira2012ordinary} that $\|\X(t) - x_*\| \le c e^{- t/c}$ for a constant $c > 0$. Similarly, $\|\Xh(t) - \hat x_*\| \le \hat c e^{- t/\hat c}$ for a constant $\hat c > 0$, and it can be further shown that $\hat c$ is bounded away from 0 under our working conditions. By making $c$ smaller if needed, we may assume that $\|\Xh(t) - \hat x_*\| \le c e^{- t/c}$ a.s.. Therefore, $\|\X(t) - x_*\| \le c/n$ and $\|\Xh(t) - \hat x_*\| \le c/n$ for all $t \ge c \log n$ a.s.. As a consequence, if we let $\Y(t) = \X(t+c\log n)$ and $\Y(t) = \Xh(t+c\log n)$, defined for $t \ge 0$, we have that $\sqrt{n h^{d+2}}(\Yt - \Y)$ converges in distribution to the Gaussian process on $[0,\infty)$ constant at $Z$ --- a degenerate limit.

It is an open question to derive a nontrivial asymptotic limit distribution for $\|\Xh - \X\|_\infty$ after proper scaling. If it is entirely dominated by what happens at $t = \infty$, the proper scaling might be in $\sqrt{n h^{d+2}}$. In any case, the corresponding confidence region is bound to be inaccurate (very wide) at finite times, and the consideration of a weighting scheme as in \secref{spherically weighted}, for example, would appear to be particularly important. We leave the construction of such a confidence region covering the entire infinite time horizon for future work.

\subsection{Bias correction and under-smoothing}

It is well known that bias must be handled carefully when using the bootstrap to construct nonparametric confidence bands or regions. Two common strategies are under-smoothing and explicit bias correction~\cite{hall1992effect}. The method studied in \secref{plugin} follows the under-smoothing approach, where a bandwidth smaller than the optimal one is chosen so that the bias decreases faster than the stochastic error. In contrast, the bias-correction approach allows the bias to be of the same order of (or even of higher order than) the stochastic term. The bias is then estimated and subtracted before applying the bootstrap. Below, we outline how the bias-correction method can be applied to construct confidence bands for gradient lines, providing an alternative to the under-smoothing strategy.

Suppose that $m_0 -3 \ge k_0+1$ and $\int_{\bbR^d} K(x) x^a dx \neq 0$ for some $a$ such that  $|a| = k_0+1$. Then it follows from the calculation in \eqref{EVHRate} that
\begin{align}
h^{-(k_0+1)}\big\{\bbE[\nabla \fh (\X(s))] - \nabla f(\X(s))\big\} \to \cB(\X(s)), \quad \text{as } h \to 0,
\end{align}
where
\begin{align}
\cB(x) := \int_{\ball(0,1)} \nabla K(z)\Big[\sum_{|a| = k_0+2} \tfrac1{a!} (-1)^{|a|} \partial^a f(x) z^{a}  \Big]dz.
\end{align}
Following analogous arguments in the proof of \thmref{weak}, we can derive that 
\[h^{-(k_0+1)}\E[\Vt(t)] \to \int_0^t \cB(\X(s)) ds, \quad \text{uniformly in $t\in[0,T]$,}\] 
and the sequence of function $h^{-(k_0+1)}\E[\Xh-\X]$ converges in the space $\cC[0,T]$ to $W$ defined by the following ODE
\begin{align}
dW(t) = \nabla^2 f(\X(t)) W(t)dt + \cB(\X(t)) dt,
\end{align} 
with $W(0) = 0$. In other words, $h^{(k_0+1)} W(t)$ is the first-order approximation of the bias of the stochastic process $\Xh-\X$. Here we use the (weaker) assumption $nh^{d+1+2(k_0+3)}\to0$ to replace $n h^{d+1+2[(m_0-1) \wedge (k_0+1)]} \to 0$ in \thmref{weak}. Then a similar proof for that theorem can be used to show that the sequence of stochastic process $\sqrt{nh^{d+1}}[\Xh-\X - h^{(k_0+1)} W]$ converges weakly in the space $\cC[0,T]$ to the Gaussian process $\U$. Let $\wh W$ be a plugin estimator of $W$ satisfying 
\begin{align}
d\wh W(t) = \nabla^2 \fh(\wh\X(t)) \wh W(t)dt + \cB(\wh\X(t)) dt,
\end{align} 
with $\wh W(0) = 0$. Here $\wh W$ can be based on a different bandwidth, say $h_1$, such that $\|W-\wh W\|_T = O_{a.s.}(h^2)$, which is possible using the arguments in \prpref{decompo}. Let $\hat p_{1-\alpha}$ be the $(1-\alpha)$-quantile of $\|\Xh^* - \Xh - h^{(k_0+1)}\wh W\|_T$. Then the following is a confidence region for $\X$ with asymptotic level $1-\alpha$ using the bias-correction approach:
\begin{align}
\wh\cR_{1-\alpha}^{\rm bc} = \Big\{\Y \in \cC[0,T] : \|\Xh - \Y - h^{(k_0+1)}\wh W \|_T \le \hat p_{1-\alpha}\Big\}.
\end{align}

\subsection{Convergence rate}
By \eqref{weak supnorm}, the convergence rate of the kernel density plugin estimator \eqref{odemodel hat} when using bandwidth $h$ is exactly $\sqrt{n h^{d+1}}$ as long as $h = h_n$ satisfies the conditions of \thmref{weak}, that is, $nh^{d+4}/\log n\to \infty$ and $n h^{d+1+2[(m_0-1) \wedge (k_0+1)]} \to 0$. Since this is in the hands of the analyst, suppose the kernel function $K$ kills $k_0 \ge m_0-2$ moments. In that case, the maximum order of magnitude for $h$ is $n h^{d+m_0-1} \to 0$, say $h = o(1) n^{-1/(d+m_0-1)}$ with $o(1)$ representing a term going to zero arbitrarily slowly, yielding a convergence rate of $o(1) n^a$ with exponent $a = \frac12 (m_0-2)/(d+m_0-1)$. (As expected, there is a standard curse of dimensionality, which simply comes from estimating the density based solely on smoothness assumptions.) 

Two questions arise: 
\begin{enumerate}
\item {\em Is this rate optimal for the present approach?} The limit distribution in \eqref{weak supnorm} is non-degenerate, so that the rate is indeed $\sqrt{n h^{d+1}}$ as long as $h$ satisfies the stated conditions. However, it might be possible to relax the upper bound on $h$, that is, $h \ll n^{-1/(d+m_0-1)}$. 
\item {\em Is this rate minimax optimal?} The question of minimax optimality for the estimation of \eqref{odemodel} is not completely trivial to formulate as the class of densities to consider needs to be carefully defined. Without going into any details here, we simply state that, although kernel density estimation is, under our working assumptions, minimax optimal  for estimating the density and its derivatives, it is not at all clear that the plugin approach in \eqref{odemodel hat} is minimax optimal. 
\end{enumerate}

We note that a rate of convergence is derived in \cite{arias2016estimation} for the entire infinite time horizon, although it is clearly suboptimal.

\subsection{Impact of discretization}
In practice, the estimated flow \eqref{odemodel hat} needs to be approximated using a numerical scheme. In \aspref{f}, $f$ is taken to be $m_0$ continuously differentiable with bounded derivatives of order $\le m_0$. If we choose a kernel function satifying \aspref{K} with $m_1 \ge m_0$, $\fh$ satisfies the same properties almost surely. In that case, $\Xh$ in \eqref{odemodel hat} is also $m_0$ times continuously differentiable with bounded derivatives of order $\le m_0$, and using the Adams–Bashforth method of order $m_0-1$ (see, e.g., \cite[Sec 6.2.1]{gautschi2012numerical}), results in a global error of order $O(\nu^{m_0-1})$ if $\nu$ is the step size: If $(x_{k,\nu} : k = 0, \dots, \lfloor T/\nu \rfloor)$ denotes the sequence that the Adams–Bashforth method returns, then $\|x_{k,\nu} - \Xh(k \nu)\| \le C_1 \wh\kappa_{m_0} \nu^{m_0-1}$ for all $k = 0, \dots, \lfloor T/\nu \rfloor$ for some constant $C_1$ that depends only on $m_0$ and $d$ (see, e.g., \cite[Th 6.1.3]{gautschi2012numerical}). We can, and do, take $x_{k,n} = x_0$, since the origin point is known.

Exercising the axiom of choice, let $\Xt$ be a $m_0$ continuously differentiable function on $[0,T]$ satisfying $\max_k \|\Xt(k\nu) - x_{k,\nu}\| \le 2 C_1 \kappa_{m_0} \nu^{m_0-1}$ such that its $m_0$ differential has smallest supnorm. In particular, eventually, $\|\nabla^{m_0} \Xt\|_T \le \|\nabla^{m_0} \Xh\|_T$, since $\Xh$ is a competitor as soon as $\wh\kappa_{m_0} \le 2 \kappa_{m_0}$, which happens eventually. (For a more constructive way of obtaining a function with similar properties, see \cite{fefferman2009fitting2}.) 
Note that there is a constant $C_2$ depending only on $m_0$ ad $d$ such that $\|\nabla^{m_0} \Xh\|_T \le C_2 \max_{m \le m_0} \wh\kappa_m^{m_0-1}$, and since under our assumptions $\wh\kappa_m \to \kappa_m$ for all $m \le m_0$, $\|\nabla^{m_0} \Xh\|_T \le C_3$ for some constant $C_3$ that also depends on $\max_{m \le m_0} \kappa_m$. 
Using \lemref{interpolation} below, we find that $\|\Xt - \Xh\|_T \le C_4 \nu^{m_0-1}$, where $C_4$ depends on $m_0$, $d$, and $\max_{m \le m_0} \kappa_m$. Therefore, if the step size $\nu = \nu_n$ is chosen small enough that $\sqrt{n h^{d+1}} \nu^{m_0-1} \to 0$, we have $\sqrt{n h^{d+1}} (\Xt - \X) \weak U$, that is, $\Xt$ behaves exactly like $\Xh$ asymptotically.

\begin{lemma}\label{lem:interpolation}
Consider $g : [0,1] \to \bbR$ which is $m$ times continuous differentiable and such that $|g(k/r)| \le \eta$ for all $k = 0, \dots, r$. Then there is a constant $C$ depending only on $m$ such that $\|g\|_\infty \le C (\eta + \|g^{m)}\|_\infty r^{-m})$. 
\end{lemma}

\begin{proof}
Let $t_k = k/r$ and $y_k = g(k/r)$.
For $t \in [0,1]$, let $p$ be the Lagrange polynomial of degree $\le m$ interpolating $(t_{k_j}, y_{k_j})$ where the $t_{k_j}$ are the $m+1$ closest grid points to $t$ (ties broken arbitrarily). Let $I = [\min_j t_{k_j}, \max_j t_{k_j}]$. 
By the Lagrange remainder formula and the fact that the spacing is $1/r$,
\[
\|g - p\|_I \le C_1 r^{-m} \|g^{(m)}\|_{I},
\]
for $C_1$ depending only on $m$.
Now, $\|p\|_I \le C_2 \eta$ where $C_2$ is the Lebesgue constant for polynomial interpolation at $m+1$ equispaced points.
Thus, we have 
\begin{align}
|g(t)| 
&\le \|g\|_I \le \|p\|_I + \|g-p\|_I \\
&\le C_2 \eta + C_1 r^{-m} \|g^{(m)}\|_{I}  \\
&\le C_2 \eta + C_1 r^{-m} \|g^{(m)}\|_\infty,
\end{align}
completing the proof since $t$ is arbitrary.
\end{proof}

\subsection{Confidence bands}
\label{sec:bands}
When seeing $\X$ on $[0,T]$ as a continuous function on that interval with values in $\bbR^d$, a {\em confidence region} for $\X$ is most directly interpreted to mean a subset of $\cC[0,T]$. This is how we used that term in \secref{plugin} and \secref{bootstrap}.
However, the term {\em confidence band} is often used in similar contexts. This comes from visualizing trace on the underlying Euclidean space $\bbR^d$ of the subset of functions constituting the confidence region. For example, for a confidence region of the form
\begin{align}\label{conf region}
\wh\cR_{1-\alpha} = \Big\{\Y \in \cC[0,T] : \|\Xh - \Y\|_T \le \hat r_{1-\alpha}\Big\},
\end{align}
its trace is the following domain of $\bbR^d$
\begin{align}\label{conf band plugin}
\wh\cB_{1-\alpha} 
&= \Big\{y \in \bbR^d : \inf_{t \in [0,T]} \|\Xh(t) - y\| \le \hat r_{1-\alpha}\Big\} \\
&= \bigcup_{t \in [0,T]} \overline\ball\big(\Xh(t), \hat r_{1-\alpha}\big),
\end{align}
which can be described as a `band' around the estimated gradient ascent line $\Xh([0,T])$ --- sometimes also called a tubular neighborhood, here of radius $\hat r_{1-\alpha}$. 

If the region \eqref{conf region} has exact level $1-\alpha$ asymptotically for the gradient ascent flow $\X$, it implies that the band \eqref{conf band plugin} has level $1-\alpha$ asymptotically for the gradient line $\X([0,T])$, or in formula,
\begin{align}\label{conf band plugin level}
\liminf_{n \to \infty} \P\big(\X([0,T]) \in \wh\cB_{1-\alpha}\big) \ge 1-\alpha.
\end{align}
In particular, the confidence band may be conservative in the large-sample limit. 

\begin{remark}
As a point of curiosity, it appears possible to build a confidence band with exactly the desired level asymptotically. 
For that, define the band around $\Xh([0,T])$ of thickness $r$ as 
\begin{align}
\wh\cB(r) = \Big\{y \in \bbR^d : \inf_{t \in [0,T]} \|\Xh(t) - y\| < r\Big\} = \bigcup_{t \in [0,T]} \ball(\Xh(t), r).
\end{align}
For $r > 0$, define the time of the first exit from the band
\[\hat\tau_r = \inf\big\{t \ge 0: \U(t) \notin \wh\cB(r)\big\}.\] 
Define 
\[\tilde\psi_r(s,y) = \P_\U\big(\hat\tau_r > T \mid \U(s) = y\big),\]
where the probability is only with respect to $\U$, with everything else, in particular $\wh\cB(r)$, being held fixed.
Then, as in \eqref{backward equation}, $\tilde\psi_r$ satisfies the following backward equation
\begin{align}
-\tilde\psi_r'(s, y) = y^\top \nabla^2 f(\X(s)) \nabla\tilde\psi_r(s, y) + \frac{1}{2} \trace \left[ B(s)^2 \, \nabla^2\tilde\psi_r(s, y) \right],
\end{align}
for $0 \le s < T$, with terminal conditions $\tilde\psi_r(s, y) = 0$ for $y \in \partial\wh\cB(r)$ and $\tilde\psi_r(T, y) = 1\{y \in \ball(0,q)\}$. 
The idea would be to choose $r$ such that $\tilde\psi_r(0, 0) = 1-\alpha$.
Although $\tilde\psi_r$ is not directly available as it depends on the distribution of $\U$, we could again estimate it by plugin, specifically, by $\hat\psi_r$, solution to
\begin{align}
-\hat\psi_r'(s, y) = y^\top \nabla^2 \fh(\Xh(s)) \nabla\hat\psi_r(s, y) + \frac{1}{2} \trace \left[ B(s)^2 \, \nabla^2\hat\psi_r(s, y) \right].
\end{align}
(A bootstrap approach seems possible as well.)
One would then let $\hat r$ be such that $\hat\psi_r(0, 0) = 1-\alpha$, and return the confidence band $\wh\cB({\hat r})$ --- or its closure. 
\end{remark}

\subsection{Matrix weighted process}
\label{sec:matrix weighted}
We discuss here in less detail a weighting scheme more elaborate than that of \secref{spherically weighted} that should be intuitively understood as being closer to Studentizing $\Xh(t) - \X(t)$. This takes place at the level of \eqref{DecomPart1}. Indeed, denote 
\begin{align}
G_\ddag(t) 
&:= \frac{1}{h^{d+1}} \bbE\big[\cE(p_t)\cE(p_t)^\top\big] \\
&= \frac{1}{h^{d+1}}\int_0^t \int_0^t \bbE\bigg[\nabla K\bigg(\frac{\X(s)-X}{h}\bigg) \nabla K\bigg(\frac{\X(u)-X}{h}\bigg)^{\!\!\top} \bigg] duds,
\end{align}
and let $B_\ddag(t)$ be symmetric satisfying  
\begin{align}
B_\ddag(t)^2 = G_\ddag'(t) = H_\ddag(t) + H_\ddag(t)^\top,
\end{align}
with 
\begin{align}
H_\ddag(t) :=
\frac{1}{h^{d+1}}\int_0^t \bbE\bigg[\nabla K\bigg(\frac{\X(t)-X}{h}\bigg) \nabla K\bigg(\frac{\X(s)-X}{h}\bigg)^{\!\!\top}\bigg] ds. 
\end{align}
Finally, let
\begin{align}
\Sigma_\ddag(t) :=  M(t) \bigg[\int_0^t M(s)^{-1} B_\ddag(s)^2 M(s)^{-\top} ds \bigg] M(t)^\top.
\end{align}
Following derivations analogous to those below \eqref{bbE E E}, for any fixed $t\in(0,T]$, we can show $B_\ddag^2(t)\to B^2(t)$ as $h\to0$, and hence $\Sigma_\ddag(t)^{-1/2} \to \Sigma(t)^{-1/2}$, which then implies 
\begin{align}
\sqrt{nh^{d+1}}\Sigma_\ddag^{-1/2}(\Xh-\X) \longto \Sigma^{-1/2} U,
\end{align}
weakly in $\cC [\eps,T]$ for any $\eps > 0$.
In addition, when $t\to0$, it can also be shown using similar arguments that
\begin{align}
\label{C0}
B_\ddag(t)^2 
\sim \frac{2t}{h} f(x_0) C_0, \quad \text{where } 
C_0 := \int_{\bbR^d} \nabla K(z) \nabla K(z)^\top dz,
\end{align}
which then implies that
$\Sigma_\ddag(t) 
\sim (t^2/h) C_0.$
Hence, applying a Taylor development, we have
\begin{align}
\sqrt{nh^{d+1}} \Sigma_\ddag(t)^{-1/2} (\Xh(t)-\X(t))
&\sim \sqrt{nh^{d+1}} \big((t^2/h) f(x_0) C_0\big)^{-1/2} \big(x_0 + t \nabla \fh(x_0) - x_0 -t\nabla f(x_0)\big) \\
&= \sqrt{nh^{d+2}} f(x_0)^{-1/2} C_0^{-1/2} (\nabla \fh(x_0) - \nabla f(x_0)) \\
& \weak \cN(0, I),
\end{align}
as is well-known. In other words, $\sqrt{nh^{d+1}}\Sigma_\ddag(t)^{-1/2}(\Xh(t)-\X(t))$ is Studentized for both fixed $t\in(0,T]$ and $t\to0$. However, $\Sigma^{-1/2}(t)U(t)$ is not the limiting process because $\Sigma^{-1/2}(t)U(t)$ is ill-defined when $t\to0$ by observing that $U(t) = B(0) W(t) + o_P(\sqrt{t})$ and $\Sigma(t) \sim t B(0)^2$ as $t\to0$, with $B(0)$ singular. We introduce an $h$-dependent Gaussian process $\U_\ddag$ satisfying $\U_\ddag(0)=0$ and  
\begin{align}
d\U_\ddag(t)
= \nabla^2 f(\X(t))\U_\ddag(t)dt + B_\ddag(t) dW(t).
\end{align}
It can be verified that $\U_\ddag(t)$ has mean zero and covariance matrix $\Sigma_\ddag(t)$, and $\sqrt{nh^{d+1}}\Sigma_\ddag(t)^{-1/2}(\Xh(t)-\X(t))$ and $\Sigma_\ddag(t)^{-1/2}\U_\ddag(t)$ have the same asymptotic distribution, both when $t$ is fixed and $t\to0$. Note that $\U_\ddag \weak U$ because $B_\ddag^2(t)\to B^2(t)$ as $h\to0$. 

To account for the Law of the Iterated Logarithm (LIL) behavior of the Gaussian processes (see \eqref{Wt unbounded} and \eqref{Wt bounded}), we introduce the deterministic weight function $\rho_0 : [0,T] \to (0, \infty)$ which is continuous and such that $\rho_0(t) \gg \sqrt{\log\log(1/t)}$ as $t \to 0$. By following analogous arguments in the proof of Theorem \ref{thm:weighted_convergence}, we believe one can show that $\sqrt{nh^{d+1}}\rho_0^{-1}\Sigma_\ddag(t)^{-1/2}(\Xh(t)-\X(t))$ and $\rho_0^{-1}\Sigma_\ddag(t)^{-1/2}\U_\ddag(t)$ converge weakly in the space $\cC[0,T]$ to the same limit, and by the Continuous Mapping Theorem, $\sqrt{nh^{d+1}}\|\rho_0^{-1}\Sigma_\ddag(t)^{-1/2}(\Xh(t)-\X(t))\|_T$ and $\|\rho_0^{-1}\Sigma_\ddag(t)^{-1/2}\U_\ddag(t)\|_T$ have the same weak convergence limit.

Let $\widehat\Sigma_\ddag(t)$ be the plugin estimator of $\Sigma_\ddag(t)$. The matrix weighted version of a confidence band is 
\begin{align}
\wh\cR_{1-\alpha}^{\rm plugin} = \Big\{\Y \in \cC[0,T] : \|\rho_0^{-1}\widehat \Sigma_\ddag^{-1/2}(\Xh - \Y)\|_{T} \le \hat q_{1-\alpha}/\sqrt{n h^{d+1}}, \; \Y(0)= x_0\Big\},
\end{align}
where $\hat q_{1-\alpha}$ now denotes the $(1-\alpha)$-quantile of $\|\rho_0^{-1}\wh\Sigma_\ddag^{-1/2}\Uh_\ddag\|_{T}$ with $\Uh_\ddag$ satisfying $\Uh_\ddag(0)=0$ and
\begin{align}
d\Uh_\ddag(t)
= \nabla^2 \fh(\Xh(t))\Uh_\ddag(t)dt + \wh B_\ddag(t) dW(t).
\end{align}
If the smoothed bootstrap is used, then $\hat q_{1-\alpha}$ can be replaced by $\hat p_{1-\alpha}$ which is the $(1-\alpha)$-quantile of $\|\rho_0^{-1}(\wh\Sigma_\ddag^*)^{-1/2}(\Xh^* - \Xh)\|_{T}$.

\subsection*{Acknowledgements}
We are grateful to Ruth Williams for helpful discussions and pointers to the literature that led quite directly to the proof of \lemref{Q quant}.

\bibliographystyle{chicago}
\bibliography{../ref}

\appendix
\section{More on the construction of confidence regions}
\label{sec:app_conf}

\subsection{When is $\Sigma(t)$ singular?}
\label{sec:Sigma singular}

We argue that, generically, $\Sigma(t)$ is non-singular for all $t > 0$.

\begin{lemma}\label{lem:B psd}
For each $t \ge 0$, $B(t)$ is singular. In fact, if the kernel function $K$ is spherically symmetric, $B(t) = c f(\X(t))^{1/2} \|\X'(t)\| P(t)$ where $P(t) = I - \X'(t) \X'(t)^\top/\|\X'(t)\|^2$, and $c$ is a constant that only depends on $K$.
\end{lemma}

\begin{proof}
For a direct view, fix $x \in \bbR^d$, let $v=\nabla f(x)$ and let $N$ be a $d\times (d-1)$ orthogonal matrix such that $N^\top v = 0$. 
Using the change of variables $z = s v + N w$ for $s\in\bbR$ and $w\in \bbR^{d-1}$, we have 
\begin{align}
H(v)
&= \int_{\bbR}\int_{\bbR^d} \nabla K(z) \nabla K\big(t v + z\big)^\top dz dt \\
&= \int_{\bbR^{d-1}} \int_{\bbR}\int_{\bbR} \nabla K(s v + N w) \nabla K\big(t v + s v + N w\big)^\top dt ds dw, \\
& = \int_{\bbR^{d-1}} L_v(w) L_v(w)^\top dw, \label{B psd}
\end{align}
where $L_v(w) = \int_{\bbR} \nabla K(s v + N w) ds$.
Then, for all $w$,
\begin{align}
L_v(w)^\top v
&= \int_{\bbR} \nabla K(s v + N w)^\top v\, ds 
= K(s v + N w) \big|_{s=-\infty}^{s=\infty} = 0.
\end{align}
We have thus shown that $H(v)$ is psd and singular with $H(v)^\top v = 0$. Given that, for any $t \ge 0$, $B(t) = f(\X(t))^{1/2} H(\X'(t))^{1/2}$, we have established that $B(t)$ is psd and singular.

When $K$ is spherically symmetric, say of the form $K(x) = \textsc{k}(\|x\|^2)$, then
\begin{align}
L_v(w) 
&= \int_\bbR 2 \textsc{k}'(s^2 \|v\|^2 + w^2)(sv + N w) ds \\
&= \lambda(\|w\|) \|v\| N w, \qquad \lambda(\|w\|) = 2 \int_\bbR \textsc{k}'(s^2 + w^2) ds.
\end{align}
Therefore,
\begin{align}
H(v) 
&= \int_{\bbR^{d-1}} \lambda(\|w\|) \|v\| N w \big(\lambda(\|w\|) \|v\| N w\big)^\top dw \\
&= \|v\|^2 N \bigg(\int_{\bbR^{d-1}} \lambda(\|w\|)^2 w w^\top dw\bigg) N^\top \\
&= c^2 \|v\|^2 N N^\top, \qquad c^2 = \frac1{d-1} \int_{\bbR^{d-1}} \lambda(\|w\|)^2 \|w\|^2 dw.
\end{align}
Note that $N N^\top$ is the orthoprojector onto $v^\perp$, and can be expressed as $N N^\top = I - v v^\top/\|v\|^2$.
Since $B(t) = f(\X(t))^{1/2} H(\X'(t))^{1/2}$, we can conclude.
\end{proof}

Assume that $K$ is spherically symmetric so that the second part of \lemref{B psd} applies. We investigate under what conditions $\Sigma(t)$ is nonsingular. 
Therefore, fix $t > 0$, and for $a \in \bbR^d$ define $b = M(t)^\top a$, and derive
\begin{align}
a^\top \Sigma(t) a
&= \int_0^t (M(t)^\top a)^\top M(s)^{-1} B(s)^\top B(s) M(s)^{-\top}  (M(t)^\top a) ds \\
&= 0 \iff B(s) M(s)^{-\top} b = 0,\quad \forall s \in [0,t].
\end{align}
Applying \lemref{B psd}, we have $B(s) = c f(\X(s))^{1/2} \|\X'(s)\| P(s)$, so that
\begin{align}
B(s) M(s)^{-\top} b = 0
\iff P(s) M(s)^{-\top} b = 0
\iff b \propto M(s)^\top \X'(s) = M(s)^\top M(s) v_0,
\end{align}
where $v_0 = \X'(0) = \nabla f(x_0)$, since $\X'(s) = M(s) v_0$. 
We thus arrive at the fact that $\Sigma(t)$ is singular if and only if $M(s)^\top M(s) v_0$ is constant for $s \in [0,t]$. 
We quickly note that, since $\X'(s) = M(s) v_0$,
\begin{align}
M(s)^\top M(s) v_0 \text{ constant for $s \in [0,t]$}
\implies \|\X'(s)\| \text{ constant for $s \in [0,t]$},
\end{align}
which is already non-generic. Going deeper, noting that $\frac{d}{ds} M(s)^\top M(s) v_0 = 2 M(s)^\top A(s) M(s) v_0$, and using the fact that $M(s)$ is nonsingular, we get
\begin{align}
M(s)^\top M(s) v_0 \text{ constant for $s \in [0,t]$}
&\iff A(s) M(s) v_0 = 0 \text{ for $s \in [0,t]$} \\
&\iff A(s) \X'(s) = 0 \text{ for $s \in [0,t]$} \\
&\iff \X''(s) = 0 \text{ for $s \in [0,t]$} \\
&\iff \X'(s) = v_0 \text{ for $s \in [0,t]$} \\
&\iff A(s) v_0 = 0 \text{ for $s \in [0,t]$},
\end{align}
since $\X'(s) = M(s) v_0$ and $\X''(s) = A(s) \X'(s)$. 
Continuing, this is equivalent to $\X(s) = x_0 + s v_0$ with $\nabla^2 f(x_0 + s v_0) v_0 = 0$ for all $s \in [0,t]$, and the second part is equivalent to $f(x_0 + s v_0) = f(x_0) + s \|v_0\|^2$ for all $s \in [0,t]$. This is an even less generic situation, but it is possible, e.g., if $f(x) = v_0^\top x$ in a neighborhood of $x_0$.  

\subsection{Behavior of $\Sigma(t)$ as $t \to 0$}
\label{sec:Sigma t=0}
We investigate the behavior of $\Sigma(t)$, defined in \eqref{Sigma}, as $t \to 0$.

In \secref{Sigma singular}, we have argued that, generically, $\Sigma(t)$ is non-singular for all $t > 0$. However, even if this is the case, some serious issues remain. Indeed, as $t\to 0$, the leading term in $\Sigma(t)$ is singular in that $\Sigma(t) \sim t B(0)^2$. And because, as $t\to 0$, 
\begin{align}\label{at 0}
\Xh(t)-\X(t) \sim t (\Xh'(0)-\X'(0)) = t (\nabla \fh(x_0) - \nabla f(x_0)),
\end{align} 
with 
\begin{align}\label{nabla at 0}
\sqrt{n h^{d+2}}(\nabla \fh(x_0) - \nabla f(x_0)) \weak \cN(0, c_1 f(x_0) I),
\end{align} 
where $c_1$ depends only on $K$ \cite{mokkadem2003law}, indicating that $\Xh(t)-\X(t)$ has a spherically symmetric distribution as $n \to \infty$ and $t\to 0$, there is the potential for $\Sigma(t)^{-1/2}(\Xh(t)-\X(t))$ to misbehave when $t\to0$.

To go further, assume $K$ is spherically symmetric. Then, using \lemref{B psd} and going through some laborious calculations, one finds that
\begin{align}
\Sigma(t) 
= c^2 t \Sigma_1 + \tfrac12 c^2 t^2 \Sigma_2 + O(t^3)
\end{align}
where $c$ is the constant of \lemref{B psd} and   
\begin{multline}
\Sigma_1 := f_0 \|v_0\|^2 P_0, \qquad 
\Sigma_2 := \|v_0\|^4 P_0 + 2 f_0 (v_0^\top H_0 v_0) I  
+ 2 f_0 \|v_0\|^2 H_0 - 2 f_0 (H_0 v_0 v_0^\top + v_0 v_0^\top H_0),
\end{multline}
with 
\begin{align}
f_0 := f(x_0), && v_0 := \nabla f(x_0), && H_0 := \nabla^2 f(x_0), && P_0 := I - v_0 v_0^\top/\|v_0\|^2.
\end{align}
It can be checked that $\Sigma_1 v_0 = \Sigma_2 v_0 = 0$, so that $\Sigma(t) v_0 = O(t^3)$ as $t\to 0$. Using an eigendecomposition of $\Sigma(t)$, we derive
\[
\|\Sigma(t)^{-1/2} v_0\| \ge d^{-1/4} \|v_0\|^{3/2} \|\Sigma(t) v_0\|^{-1/2},
\]
so that $t \Sigma(t)^{-1/2} v_0$ diverges as $t\to 0$. As we saw in \eqref{at 0}-\eqref{nabla at 0}, $\Xh(t)-\X(t) \sim t (\hat v_0-v_0)$ as $t\to 0$ with $\hat v_0-v_0$ having an approximately spherically symmetric distribution as $n \to \infty$, and so it does not bode well for $\Sigma(t)^{-1/2}(\Xh(t)-\X(t))$ as $t \to 0$ and $n \to \infty$. 
(And there is the compounding issue that $\sqrt{n h^{d+1}} \|\hat v_0-v_0\| \to \infty$ in probability as $n \to \infty$.)

\section{Proof of \thmref{weighted_convergence}}
\label{sec:weighted_convergence_proof}
Because the boundary weight $1/\rho(t)$ introduces a severe analytical singularity at the origin $t=0$, establishing stochastic equicontinuity directly on the entire space $\cC[0, T]$ is technically intractable. We resolve this by invoking the Approximation Theorem for weak convergence \cite[Th~3.2]{billingsley1999convergence}. This theorem establishes that a sequence of random elements $X_n$ in a metric space $(\cS,d)$ converges weakly to a true limit $X$ if it can be approximated by a sequence $X_{n, \delta}$ satisfying three conditions: (i) the approximator converges weakly to a truncated limit $X_\delta$ as $n \to \infty$ for any fixed $\delta > 0$, (ii) the truncated limit $X_\delta$ converges to the true limit $X$ as $\delta \downarrow 0$, and (iii) the approximation error is asymptotically negligible, meaning that for any threshold $\eta > 0$, the discrepancy satisfies 
\[\lim_{\delta \downarrow 0} \limsup_{n \to \infty} \P(d(X_{n, \delta}, X_n) > \eta) = 0.\]

As usual, we equip the space $\cC[0, T]$ with the uniform metric $d(\X,\Y) = \sup_{t \in [0, T]} \|\X(t) - \Y(t)\|$. We define the true tracking error sequence and its corresponding continuous Gaussian limit point-by-point as
$$X_n(t) := \frac{\sqrt{nh^{d+1}} (\Xh(t) - \X(t))}{\rho(t)} \quad \text{and} \quad X(t) := \frac{U(t)}{\rho(t)}.$$
To ensure the approximators remain in $\cC[0, T]$, we define them by holding their trajectories constant on the boundary interval $[0, \delta]$, matching the value at the threshold $\delta$, that is, 
$$X_{n, \delta}(t) := X_n(t \vee \delta) \quad \text{and} \quad X_\delta(t) := X(t \vee \delta).$$
Consequently, the proof proceeds in three distinct phases. 
First, we establish the weak convergence $X_{n, \delta} \longweak X_\delta$ on the compact subinterval $[\delta, T]$, where the weight $\rho(t)$ is safely bounded away from zero.
Second, we prove the asymptotic negligibility condition by demonstrating that the approximation error vanishes in probability using the bound
\[d(X_{n, \delta}, X_n) = \sup_{t \in [0, \delta]} \|X_n(\delta) - X_n(t)\| \le 2 \sup_{t \in (0, \delta]} \|X_n(t)\|.\]
Finally, we demonstrate that $X_\delta \longweak X$ as $\delta \downarrow 0$, formally stitching the interior and boundary regimes together to conclude.

\paragraph{Step 1: Weak convergence on the compact subinterval}
By \thmref{weak}, the unweighted standardized tracking error converges weakly to the continuous Gaussian process: $\sqrt{nh^{d+1}} (\Xh - \X) \longweak U$ in $\cC[0, T]$. For any fixed $\delta > 0$, the truncated weight function $\rho(t \vee \delta)$ is bounded away from zero. Consequently, the functional mapping $y(t) \mapsto y(t \vee \delta) / \rho(t \vee \delta)$ is globally continuous on the space $\cC[0, T]$. Therefore, applying the Continuous Mapping Theorem implies that
\begin{equation}
\frac{\sqrt{nh^{d+1}} (\Xh(t \vee \delta) - \X(t \vee \delta))}{\rho(t \vee \delta)} \longweak \frac{U(t \vee \delta)}{\rho(t \vee \delta)}, \quad \text{as } n \to \infty,
\end{equation}
as processes in $\cC[0, T]$.
This formally satisfies Condition (i) of the Approximation Theorem.

\paragraph{Step 2: Remainder bound}
From the proof of \prpref{decompo}, the remainder satisfies 
\[\|\Dt(t)\| \leq \Delta_n t \|\Xh - \X\|_t,\]
for all $t \le T$, for some $\Delta_n = o_{\P}(1)$, from which we deduce
\begin{equation}
\|\Dt(t)\|/\rho(t) \leq \Delta_n \big[t \|\rho\|_\infty/\rho(t)\big]\, \|(\Xh - \X)/\rho\|_t,
\end{equation}
and taking the supremum over $(0, \delta]$, we obtain
\begin{equation}
\|\Dt/\rho\|_\delta \leq \wt\Delta_n \|(\Xh - \X)/\rho\|_\delta,
\end{equation}
with $\wt\Delta_n = o_{\P}(1)$, using the fact that $t/\rho(t)$ is bounded on $[0,T]$.
Recall the definition of $\Ut$ in \eqref{Ut} and use the triangle inequality to deduce that
\begin{equation}
    (1 - \wt\Delta_n) \sqrt{nh^{d+1}}\|(\Xh - \X)/\rho\|_\delta \leq \|\Ut/\rho\|_\delta.
\end{equation}
Hence, to establish that
\begin{align}
\lim_{\delta \downarrow 0} \limsup_{n \to \infty} \P\Big(\sqrt{nh^{d+1}}\|(\Xh - \X)/\rho\|_\delta > \eta\Big) = 0, \quad \forall \eta > 0,
\end{align}
it suffices to show that 
\begin{align} \label{Ut need to show}
\lim_{\delta \downarrow 0} \limsup_{n \to \infty} \P\Big(\|\Ut/\rho\|_\delta > \eta\Big) = 0, \quad \forall \eta > 0.
\end{align}

\paragraph{Step 3: Integral representation and empirical process form}
From \prpref{decompo}, $\Ut$ satisfies
\begin{equation}
    \frac{d}{dt} \Ut(t) = A(t) \Ut(t) + \sqrt{nh^{d+1}} \big[\nabla\hat{f}(\X(t)) - \nabla f(\X(t))\big],
\end{equation}
with the initial condition $\Ut(0) = 0$.
One can verify that 
\begin{equation}
    \Ut(t) = \sqrt{nh^{d+1}} \int_0^t \Phi(t, s) \big[\nabla\hat{f}(\X(s)) - \nabla f(\X(s))\big] ds,
\end{equation}
where
\begin{align}
\Phi(t,s) := M(t) M(s)^{-1}, \quad \text{with $M$ defined in \eqref{A}.}
\end{align}

By the fact that $A$ and $M$ in \eqref{A} are continuous and, for $M$, non-singular, on the whole of $[0,T]$, $\Phi$ is bounded on $[0,T]^2$. Hence, combining with \eqref{EVHRate}, we get the following bound on the mean of $\Ut$
\begin{align}
\|\E\Ut(t)\|
&\le \sqrt{nh^{d+1}} \int_0^t \|\Phi(t, s)\|\, \big\|\nabla\hat{f}(\X(s)) - \nabla f(\X(s))\big\| ds \\
&\le \sqrt{nh^{d+1}}\, \|\Phi\|_{T,T}\, \sup_{x \in \bbR^d} \|\nabla\hat{f}(x) - \nabla f(x)\big\|\, t \\
&\le \sqrt{nh^{d+1}}\, \|\Phi\|_{T,T}\, C_1 h^{(m_0-1) \wedge (k_0+1)}\, t \\
&= C_2 \sqrt{nh^{d+1+ 2 [(m_0-1) \wedge (k_0+1)]}}\, t 
=: \Delta_n t,
\end{align}
where $\Delta_n = o(1)$ by our assumptions on $h$. 
In particular, since the above is uniform in $t \in [0,T]$,
\begin{equation}\label{global mean}
\|\E \Ut /\rho\|_T \le \Delta_n \sup_{t \in (0, T]} (t/\rho(t)) = o(1),
\end{equation}
by our assumptions on $\rho$.

Next we analyze the centered process 
\[\Ut(t) - \E \Ut(t) = \sqrt{nh^{d+1}} \big[\widehat{\cJ}(P_t) - \E\widehat{\cJ}(P_t) \big],\] 
where $\widehat{\cJ}$ and $\cJ$ are defined in \eqref{whJp}, and $P_t(s) = \Phi(t, s)\1_{[0,t]}(s)$. 
We do so separately on $(0, t_n]$ and $[t_n, \delta]$, for a carefully chosen sequence $t_n \to 0$.
Recall \eqref{etavar} and let
\[L_n := \sqrt{nh^{d+1}} \big\|\nabla \hat{f} - \E\nabla \hat{f}\big\| = \sqrt{n h^{d+1}} \hat\eta^{\rm sd}_1 = O_{\P}(\sqrt{\log n / h}).\]

\medskip\noindent
• \textit{Bound on $(0, t_n]$.} \quad
Using $\|\mathscr{U}v\|_t \le e^{\kappa_2 t} \|v\|_t$ for any $v\in \cC_0[0,T]$, we have 
$$\|\Ut(t) - \E \Ut(t)\| \le e^{\kappa_2 t} \|\Vt(t) - \E \Vt(t)\| \le t e^{\kappa_2 t} L_n.$$ 
Dividing by $\rho(t)$ yields 
\begin{align}\label{extreme centered}
\sup_{t \in (0, t_n]} \frac{\|\Ut(t) - \E\Ut(t)\|}{\rho(t)} \le  e^{\kappa_2 t_n} L_n \sup_{t \in (0, t_n]} \frac{t}{\rho(t)} = O(L_n \sqrt{t_n}) = O_{\P}\left(\sqrt{\frac{t_n \log n}{h}}\right).
\end{align} 
We impose $t_n \ll h / \log n$, so that this term converges to zero in probability. 

\medskip\noindent
• \textit{Bound on $[t_n, \delta]$.} \quad
For $t\in [t_n, \delta]$, we can write 
\begin{equation}
    \frac{\Ut(t) - \E \Ut(t)}{\rho(t)} = \frac{1}{\sqrt{n}} \sum_{i=1}^n \Big( \beta_t(X_i) - \E[\beta_t(X_i)] \Big) =: \bbG_n(\beta_t),
\end{equation}
where 
\[\beta_t(x) := \frac{1}{\rho(t)\sqrt{h^{d+1}}} \int_0^t \Phi(t, s) \nabla K\left(\frac{x - \X(s)}{h}\right) ds.\]
Let $\cB_n := \{\beta_t : t \in [t_n, \delta] \}$, noting that the bandwidth $h \equiv h_n$ is a deterministic sequence driven entirely by the sample size $n$. (Consequently, when evaluating the metric entropy for any fixed $n$, the bandwidth $h$ acts as a fixed constant.) The class $\cB_n$ is therefore a one-parameter family of functions indexed by $t \in [t_n, \delta]$.
Define the supremum of this empirical process 
\begin{equation}
    Z_n^\delta := \sup_{t \in [t_n, \delta]} \frac{\|\Ut(t) - \E \Ut(t)\|}{\rho(t)} = \sup_{\beta_t \in \cB_n} \|\bbG_n(\beta_t)\|.
\end{equation}

To bound this transitional empirical process, we will apply the following result, which is an immediate consequence of \cite[Cor~2.2]{chernozhukov2014gaussian}, to couple the supremum $Z_n^\delta$ to the supremum of a centered Gaussian process $Z_{n,h}^\delta$ sharing the exact same covariance structure.

\begin{lemma}
\label{lem:Cor2.2}
Suppose $\cB$ is a class of functions from $\bbR^d$ to $\bbR^d$ that is pointwise measurable, with uniform supnorm bound $b$, and $\epsilon$-covering number with respect to the supnorm bounded by $(a_1 b/\epsilon)^{a_2}$, for some constants $b > 0$, $a_1 \ge e$, and $a_2 \ge 1$. Consider a random vector $X$ on $\bbR^d$ and define\,\footnote{The moment conditions in \cite[Cor~2.2]{chernozhukov2014gaussian} are trivially satisfied in the context of the supnorm ($q=\infty$ in \cite{chernozhukov2014gaussian}) and a constant envelope.} 
$\sigma^2 := \sup_{\beta \in \cB} \E[\|\beta(X)\|^2]$. With $X_1, \dots, X_n$ iid $X$, define the empirical process and its supremum 
\[\bbG_n(\beta) := \frac1{\sqrt{n}} \sum_{i=1}^n \big(\beta(X_i) - \E[\beta(X)]\big), \qquad Z_n := \sup_{\beta \in \cB} \bbG_n(\beta).\] 
Suppose there is a Gaussian process $G_X$ on $\cB$ such that
\[\E[G_X(\beta_1) G_X(\beta_2)] = \E[\beta_1(X) \beta_2(X)], \quad \forall \beta_1, \beta_2 \in \cB.\]
Then there is a universal constant $C>0$ with the property that, for every $\gamma \in (0,1)$, there exists a random variable $\widetilde{Z}$ with distribution that of $\sup_{\beta \in \cB} G_X(\beta)$, such that
\[
\P\Bigg\{
|Z_n - \widetilde{Z}| >
\frac{b K_n}{\gamma^{1/2} n^{1/2}}
+
\frac{(b\sigma)^{1/2} K_n^{3/4}}{\gamma^{1/2} n^{1/4}}
+
\frac{(b\sigma^2 K_n^2)^{1/3}}{\gamma^{1/3} n^{1/6}}
\Bigg\}
\le
C \left( \gamma + \frac{\log n}{n} \right),
\]
where
$
K_n := C a_2 \big( \log(n) \vee \log(a_1 b/\sigma) \big).
$
\end{lemma}

We verify that the conditions of \lemref{Cor2.2} are met. 
First, we establish pointwise measurability. The index set $[t_n, \delta] \subset \bbR$ is a separable metric space containing a countable dense subset $\cT$. Because $\X$ and $\Phi$ are continuously differentiable, and $\rho$ is continuous and positive, the mapping $t \mapsto \beta_t(x)$ is continuous for every fixed $x$. Therefore, the countable subclass $\{\beta_t : t \in \cT\}$ satisfies the pointwise approximation requirement.

Next, we define a constant envelope for $\cB_n$.
For any fixed point $x$, the integrand in $\beta_t(x)$ is non-zero only on the set $\cS(x) := \big\{s \in [0, t] : \|x - \X(s)\| \le h \big\}$. 
We use arguments similar to those in \eqref{gamma} and \eqref{tau bounded}. 
Because $\X' \ne 0$ on $[0,T]$, by continuity, $v_{\rm min} := \inf_{0 \le s \le T} \|\X'(s)\| > 0$. 
Also, because $s \mapsto f(\X(s))$ is strictly increasing, the path $\X([0, T])$ has no self-intersections, so that it is a compact, simple, and smooth curve in $\bbR^d$, which therefore has a positive reach $\tau > 0$. This implies that, when $h < \tau/2$ (which happens eventually as $n\to\infty$), $\X([0, T]) \cap B(x, h)$ must be connected \cite[Rem 4.15(1)]{federer1959curvature}. Consequently, the total time spent in the intersection set is bounded by the maximum diameter of the ball divided by the minimum velocity, yielding $|\cS(x)| \le 2h / v_{\rm min}$.
Combining this with the uniformly boundedness of the state transition matrix $\Phi$, we establish an envelope $b_n$ over $\cB_n$ as follows
\begin{align}
\sup_{t \in [t_n, \delta]} \sup_{x \in \bbR^d} \|\beta_t(x)\| 
    &\le \sup_{t \in [t_n, \delta]} \frac{1}{\rho(t)\sqrt{h^{d+1}}} \int_{\cS(x)} \|\Phi(t, s)\| \cdot \left\| \nabla K\left(\frac{x - \X(s)}{h}\right) \right\| ds  \\
    &\le \frac{\|\Phi\|_{T,T} \|\nabla K\|_\infty (2 h/v_{\rm min})}{\rho_n \sqrt{h^{d+1}}} = : b_n. \label{envelope}
\end{align}
where $\rho_n := \inf_{t \in [t_n, \delta]} \rho(t)$.
Note that 
\begin{align}\label{b_n asymp}
b_n \asymp h^{(1-d)/2}/\rho_n.
\end{align}

Next, we bound the covering number. For that, we show that the class $\cB_n$ is uniformly Lipschitz with respect to the index variable $t \in [t_n,\delta]$, and bound the supremum Lipschitz constant. We first evaluate the derivative of $\beta_t$ with respect to $t$ as follows
\begin{align}
    \frac{\partial \beta_t(x)}{\partial t} &= \underbrace{-\frac{\rho'(t)}{\rho(t)} \beta_t(x)}_{\text{(I)}} 
    + \underbrace{\frac{1}{\rho(t)\sqrt{h^{d+1}}} \Phi(t, t) \nabla K\left(\frac{x - \X(t)}{h}\right)}_{\text{(II)}}  \\
    &\quad + \underbrace{\frac{1}{\rho(t)\sqrt{h^{d+1}}} \int_0^t \frac{\partial \Phi(t, s)}{\partial t} \nabla K\left(\frac{x - \X(s)}{h}\right) ds}_{\text{(III)}}.
\end{align}
Below we bound the supremum of each term over $x \in \bbR^d$ and $t \in [t_n, \delta]$. For (I), using the fact that $t \mapsto t \rho'(t) / \rho(t)$ is bounded by assumption, uniformly in $x \in \bbR^d$ and $t \in [t_n, \delta]$,
\begin{equation}
\text{(I)} = O(1/t_n) b_n = O(b_n/t_n).
\end{equation}
For (II), since $\Phi(t, t)$ is the identity matrix and $\nabla K$ has bounded gradient, uniformly in $x \in \bbR^d$ and $t \in [t_n, \delta]$,
\begin{equation}
\text{(II)} \le \frac{1}{\rho_n \sqrt{h^{d+1}}} \|\nabla K\|_\infty 
= O(b_n/h).
\end{equation}
For (III), as in \eqref{envelope}, but this time based on the fact that $\partial_t \Phi(t, s)$ is bounded, uniformly in $x \in \bbR^d$ and $t \in [t_n, \delta]$,
\begin{equation}
    \text{(III)} = O(b_n).
\end{equation}
Hence, 
\begin{align}\label{L_n}
L_n := \sup_{t \in [t_n, \delta]} \sup_{x \in \bbR^d} \left\| \frac{\partial \beta_t(x)}{\partial t} \right\| 
= O(b_n/t_n) + O(b_n/h) + O(b_n) 
= O(b_n/t_n),
\end{align}
by our choice of $t_n$ satisfying $t_n \ll h/\log n$.
It is now straightforward to bound the $\epsilon$-covering number of $\cB_n$ with respect to the supnorm. We cover $\cB_n$ with $L_\infty$-balls of radius $\epsilon$ by simply placing a one-dimensional grid over the time index set $[t_n, \delta]$ with spacing $\epsilon/L_n \asymp \epsilon t_n/b_n$. In more detail, define $s_j = j (\epsilon/L_n)$ for $j = 1, \dots, J_\epsilon$, where $J_\epsilon := \lfloor \delta/(\epsilon/L_n) \rfloor$. 
Then for any $t \in [t_n, \delta]$, there is $j$ in that range such that $|t-s_j| \le \epsilon/L_n$, implying that
\[
\|\beta_t - \beta_{s_j}\|_\infty
\le L_n |t-s_j| 
\le L_n (\epsilon/L_n) = \epsilon.
\]
Therefore, $\{\beta_{s_j} : j = 1, \dots, J_\epsilon\}$ is an $\epsilon$-covering of $\cB_n$, implying that the $\epsilon$-covering number of $\cB_n$ is bounded by $J_\epsilon \asymp L_n / \epsilon \asymp b_n/t_n \epsilon$ by way of \eqref{L_n}. 
This satisfies the requirement of \lemref{Cor2.2} with $a_1 \asymp 1/t_n$ and $a_2 = 1$. 

Third, we verify the uniform moment conditions. To establish a uniform variance bound, we express the functions in $\cB_n$ using the operator $\cE(\cdot)$ in \eqref{Eoper} as follows
\[\beta_t(X) = \frac{1}{\rho(t)\sqrt{h^{d+1}}} \cE(P_t).\] 
It is easy to see that \eqref{COvpart2} can be extended to matrix-valued function. Utilizing this, we can bound the maximum variance by
\begin{align}
\sigma_n^2 := \sup_{\beta \in \cB_n} \E\big[\|\beta(X)\|^2\big] &= \sup_{t \in [t_n, \delta]} \frac{1}{\rho(t)^2 h^{d+1}} \trace \bbE \big[\cE(P_t)\cE(P_t)^\top\big]  \\
    &\lesssim \sup_{t \in [t_n, \delta]} \frac{1}{\rho(t)^2 h^{d+1}} \cdot h^{d+1} \int_0^T \|P_t(s)\|^2 ds  \\
    &= \sup_{t \in [t_n, \delta]} \frac{1}{\rho(t)^2} \int_0^t \|\Phi(t, s)\|^2 ds  \\
    &\lesssim \sup_{t \in [t_n, \delta]} \frac{t}{\rho(t)^2} 
    \lesssim 1,
\end{align}
by our assumptions on $\rho$. 
Since the envelope $b_n \to\infty$, we have $\sigma_n^2 \le b_n$ eventually.
We can also bound the third absolute moment as follows
\begin{align}
    \sup_{\beta \in \cB_n} \E\big[\|\beta(X)\|^3\big] 
    \le \sup_{\beta  \in \cB_n} \Big( \|\beta\|_{\infty} \cdot \E\big[\|\beta(X)\|^2\big] \Big)
    \le b_n \sigma_n^2.
\end{align}

All this confirms that all the conditions in \lemref{Cor2.2} are met.

In \lemref{Cor2.2}, the supremum of the empirical process $Z_n^\delta$ is coupled to the supremum of a centered Gaussian process with the same covariance structure. Next we identify the coupled Gaussian supremum. Let $$g_t(X) = h^{-(d+1)/2} \int_0^t \Phi(t, u) \nabla K((X - \X(u))/h) du.$$ 
Let $U_h(t)$ be a centered Gaussian process with covariance function 
\[\Sigma_h(s, t) := \E[g_s(X) g_t(X)^\top] - \E[g_s(X)] \E[g_t(X)]^\top,\] 
which is also the covariance of the unweighted process $\Ut(t) - \E \Ut(t)$. Let 
\[Z_{n,h}^\delta = \sup_{t \in [t_n, \delta]} \frac{\|U_h(t)\|}{\rho(t)}.\]
By \lemref{Cor2.2}, there exist a constant $C > 0$ and a coupled random variable $\widetilde{Z}_{n,h}^\delta$ distributed identically to $Z_{n,h}^\delta$ such that, for any $\gamma \in (0,1)$,
\begin{equation}
    \P\big(|Z_n^\delta - \widetilde{Z}_{n,h}^\delta| > C r_{n, \gamma}\big) \le C\left(\gamma + \frac{\log n}{n}\right),
\end{equation}
where 
\begin{equation}
    r_{n, \gamma} := \frac{b_n Q_n}{\gamma^{1/2} n^{1/2}} + \frac{b_n^{1/2} Q_n^{3/4}}{\gamma^{1/2} n^{1/4}} + \frac{b_n^{1/3} Q_n^{2/3}}{\gamma^{1/3} n^{1/6}}, \quad \text{with (here) } Q_n := \log(n) \vee \log(b_n/t_n).
\end{equation}
We set $\gamma = \gamma_n := (\log n)^{-1}$ and below we will choose $t_n \to 0$ at most polynomially in $n$, and given that $h = h_n \to 0$  at most polynomially in $n$ and the order of magnitude of $b_n \to\infty$ in \eqref{b_n asymp}, $Q_n \asymp \log n$, so that
\begin{equation}
    r_n := r_{n, \gamma_n} = O(\log n)^{3/2} \left(\frac{b_n}{n^{1/2}} + \frac{b_n^{1/2}}{n^{1/4}} + \frac{b_n^{1/3}}{n^{1/6}}\right).
\end{equation}
Thus, to guarantee that $r_n \to 0$, it suffices that $(\log n)^{3/2} b_n^{1/3}/n^{1/6} \to 0.$ 
By \eqref{b_n asymp} and our assumptions on $\rho$, it is the case that $b_n \ll h^{(1-d)/2}/t_n^{1/2}$, so that it is enough that $t_n \gg h^{1-d} n^{-1} (\log n)^9$. (Given that $h = h_n \to 0$  at most polynomially in $n$ by assumption, this already implies that $t_n \to 0$ at most polynomially in $n$, as announced.)

\medskip\noindent
• \textit{Bound on $(0, \delta]$.} \quad
By the union bound, for any $\eta>0$, 
\begin{align}
\P\big(\|\Ut/\rho\|_\delta > \eta\big) 
\le \P\big(\|(\Ut - \E \Ut)/\rho\|_\delta > \eta/2\big) + \P\big(\|\E \Ut/\rho\|_\delta > \eta/2\big). 
\end{align}
In view of \eqref{global mean}, for any $\delta \le T$,
\begin{align}
    \limsup_{n \to \infty} \P\big(\|\E \Ut/\rho\|_\delta > \eta/2\big)
    \le \limsup_{n \to \infty} \P\big(\|\E \Ut/\rho\|_T > \eta/2\big)
= 0.
\end{align}
For the other term, by the union bound,
\begin{align}
\P\big(\|(\Ut - \E \Ut)/\rho\|_\delta > \eta/2\big)
\le \P\big(\|(\Ut - \E \Ut)/\rho\|_{[0,t_n]} > \eta/2\big)
+ \P\big(\|(\Ut - \E \Ut)/\rho\|_{[t_n, \delta]} > \eta/2\big).
\end{align}
As established in \eqref{extreme centered}, under the condition $t_n \ll h / \log n$, 
\begin{align}
\lim_{n \to \infty} \P\big(\|(\Ut - \E \Ut)/\rho\|_{[0,t_n]} > \eta/2\big) = 0.
\end{align}
For the other term, by the union bound,
\begin{align}
\P\big(\|(\Ut - \E \Ut)/\rho\|_{[t_n, \delta]} > \eta/2\big) 
&= \P\big(Z_n^\delta > \eta/2\big) \\
&\le \P\left(\widetilde{Z}_{n,h}^\delta > \eta/2 - r_n\right) + o(1),
\end{align}
with $r_n \to 0$ under the condition $t_n \gg h^{1-d} n^{-1} (\log n)^9$. 
Note that the two conditions on $t_n$ are compatible by the fact that $n h^d \to \infty$ polynomially in $n$ under our assumptions on $h$.
Consequently, eventually as $n \to \infty$, 
\begin{align}
\limsup_{n \to \infty} \P\big(\|(\Ut - \E \Ut)/\rho\|_{[t_n, \delta]} > \eta/2\big) 
\le \limsup_{h \to 0} \P(Z_h^\delta > \eta/3),
\end{align}
where (with some abuse on notation)
\begin{align}
Z_h^\delta := \sup_{t \in (0, \delta]} \frac{\|U_h(t)\|}{\rho(t)}.
\end{align}

For any fixed bandwidth $h > 0$, the Gaussian process $U_h(t)$ is defined by integrating the smooth kernel $K$, ensuring its sample paths are continuously differentiable with $U_h(0) = 0$ almost surely. Since the boundary weight satisfies $\rho(t) \gg \sqrt{t}$, the ratio $\|U_h(t)\|/\rho(t)$ continuously vanishes at the origin, guaranteeing that the supremum on the right-hand side is a well-defined, finite random variable.

To establish its asymptotic limit, we recall that \thmref{weak} established that $\Ut \weak U$ in $\cC[0, T]$. By standard Gaussian approximation theory, the tightness of an empirical process implies the tightness of the centered Gaussian process sharing its exact covariance function. Therefore, $(U_h)$ is tight in $\cC[0, T]$ and converges weakly to the same process, i.e., $U_h \weak U$ (as $h \to 0$).

Because the functional $u \mapsto \sup_{t \in (0, \delta]} \|u(t)\|/\rho(t)$ is continuous on the set of paths where $u(t)/\rho(t) \to 0$ as $t \downarrow 0$, and the Law of the Iterated Logarithm guarantees the limiting process $U$ resides in this set almost surely, the Continuous Mapping Theorem yields
\begin{equation}
Z_h^\delta = \sup_{t \in (0, \delta]} \frac{\|U_h(t)\|}{\rho(t)} \longweak Z^\delta := \sup_{t \in (0, \delta]} \frac{\| U(t) \|}{\rho(t)}, \quad h \to 0.
\end{equation}
Consequently, by the Portmanteau Theorem, 
\begin{align}
\limsup_{h \to 0} \P(Z_h^\delta > \eta/3)
\le \P(Z^\delta \ge \eta/3).
\end{align}
The Law of the Iterated Logarithm guarantees that $\lim_{t \downarrow 0} \|U(t)\|/\rho(t) = 0$ almost surely, implying that
\begin{align}
\P(Z^\delta \ge \eta/3) \to 0, \quad \delta \to 0.
\end{align}
This formally establishes the uniform asymptotic negligibility of the approximation error, successfully fulfilling Condition (iii) of the Approximation Theorem \cite[Th~3.2]{billingsley1999convergence}.

Finally, to satisfy Condition (ii) of the Approximation Theorem, we must verify that the truncated limit process converges weakly to the Gaussian limit process $U/\rho$ as $\delta \downarrow 0$. Let the truncated limit be defined by truncation at $\delta$, i.e., $t \mapsto U(t \vee \delta)/\rho(t \vee \delta)$. 
By the triangle inequality,
\begin{align}
\sup_{t \in [0, T]} \left| \frac{U(t \vee \delta)}{\rho(t \vee \delta)} - \frac{U(t)}{\rho(t)} \right| 
&= \sup_{t \in (0, \delta]} \left| \frac{U(\delta)}{\rho(\delta)} - \frac{U(t)}{\rho(t)} \right| \\
&\le 2 \sup_{t \in (0, \delta]} \frac{|U(t)|}{\rho(t)} = 2 Z^\delta \to 0,
\end{align}
in probability as $\delta \to 0$, as we saw above.

Having satisfied all three conditions of the Approximation Theorem \cite[Th~3.2]{billingsley1999convergence}, we are able to conclude.

\section{On the suprema of solutions to SDEs}
\label{sec:app_sup}

We study the supnorm of the solution to an SDE of the form
\begin{align}
\label{sde}
d\U(t) = A(t) \U(t) dt + B(t) dW(t), \quad t \in [0,T],
\end{align}
with initial condition $\U(0)=0$. As before, $W$ is a standard Brownian motion in $\bbR^d$, and we assume that $A$ and $B$ are continuous and bounded on $[0,T]$, with $B$ being psd.
Because the SDE is linear, $\U$ can expressed as follows
\begin{align}\label{U sde}
\U(t) = \int_0^t M(t) M(s)^{-1} B(s) dW(s),
\end{align}
where $M$ is solution to $M'(t) = A(t) M(t)$ with $M(0) = I$.
This process has covariance matrix
\begin{align}\label{Sigma sde}
\Sigma(t) := \Cov \U(t) = \int_0^t M(t) M(s)^{-1} B(s)^2 M(s)^{-\top} M(t)^\top ds.
\end{align}
Under \aspref{f}, the SDE \eqref{SDEMain} is certainly of this form. 

The suprema of Gaussian processes have been the object of intense study over decades. We rely below on Malliavin calculus \cite{nualart2006malliavin} to derive the particular result that we want.

\begin{lemma}
\label{lem:Q quant}
Consider the solution $\U$ to an SDE of the form \eqref{sde}.  
Then $Q := \|\U\|_T$ is absolutely continuous and its support is an interval. Consequently, its distribution function is continuous and strictly increasing over its support, or equivalently, its quantile function is continuous and strictly increasing over $(0,1)$.
\end{lemma}

\begin{proof}
It suffices to show that $Q$ has a density with respect to the Lebesgue measure and that its support is an interval. 

Based on the expression for $\U$ in \eqref{U sde}, the Malliavin derivative of $\U(t)$ is given by 
\[D_s \U(t) = M(t) M(s)^{-1} B(s)\, \1_{[0,t]}(s).\]
Therefore, by the chain rule, $H(t) := \|\U(t)\|^2$ has Malliavin derivative 
\[D_s H(t) = 2 \U(t)^\top D_s \U(t) = 2 \U(t)^\top M(t) M(s)^{-1} B(s)\, \1_{[0,t]}(s).\]
The norm of this as an element of $L^2([0,T]; \bbR^d)$ --- which is the underlying Hilbert space --- is given by 
\begin{align}
\|D H(t)\|_{L^2}^2
= \int_0^T D_s H(t)^2 ds
&= 4 \int_0^t
\U(t)^\top M(t)M(s)^{-1} B(s)B(s)^\top
M(s)^{-\top}M(t)^\top \U(t) ds \\
&= 4\, \U(t)^\top \Sigma(t) \U(t). \label{DH L2}
\end{align}

We are in a position to apply \cite[Prop 2.1.11]{nualart2006malliavin}, which gives the existence of a density for $Q^2 = \|\U\|_T^2 = \sup_{t \in [0,T]} H(t)$ --- and therefore for $Q$  --- if, with probability one, 
\begin{align}\label{DH}
\text{$\|D H(t)\|_{L^2} \ne 0$ for all $t$ such that $H(t) = Q^2$.}
\end{align}
Note that $Q > 0$ with probability one since, for example, $\U(T)$ is normal with covariance matrix $\Sigma(T) \ne 0$. Therefore, with probability one, $\U(t) \ne 0$ for all $t$ such that $H(t) = Q^2$.
In view of \eqref{DH L2}, this immediately implies \eqref{DH} if it is the case that $\Sigma(t)$ is non-singular for all $t > 0$. (Note that $\U(0) = 0$, so that, with probability one, $t = 0$ is not a maximizer of $H$.)
However, assuming this is true is not needed, as we show below in \lemref{in range}. 

It remains to show that the support of $Q$ is an interval. 
Without going into straightforward details, \cite[Prop 2.1.10]{nualart2006malliavin} applies to $H(t)$ to give that $Q^2$ has a square integrable Malliavin derivative. We then invoke \cite[Prop 2.1.7]{nualart2006malliavin}, which implies that the support of a random variable must be connected.
\end{proof}

\begin{lemma}\label{lem:in range}
Let $\U$ be a continuous stochastic process in dimension $d$ defined on the interval $(0,T)$, with well-defined and continuous covariance matrix $\Sigma$. If $\rank \Sigma$ has finitely many discontinuities, then with probability one, $\U(t) \in \range \Sigma(t)$ for all $t \in (0,T)$. 
\end{lemma}

The result applies in the context of \lemref{Q quant} because, for the solution to an SDE of the form \eqref{sde}, $\Sigma$ is increasing in the Loewner order since, for any $s < t$, 
\[\Sigma(t) - \Sigma(s) = \int_s^t M(t) M(r)^{-1} B(r)^2 M(r)^{-\top} M(t)^\top dr \succeq 0,\]
and $B(r)$ is psd for all $r$. Therefore, $\rank \Sigma$ must be non-decreasing, and thus can have at most $d$ discontinuities. 
And the result is useful because $\U(t) \in \range \Sigma(t)$ implies that $\U(t)^\top \Sigma(t) \U(t) = 0$ cannot happen unless $\U(t) = 0$.

\begin{proof}
Let $0 < t_1 < \cdots < t_m < T$ denote the possible discontinuity points, and define $t_0 = 0$ and $t_{m+1} = T$. 
Consider the projection $P(t)$ onto the null space of $\Sigma(t)$, defined in such a way that $P$ is continuous on each interval $(t_j, t_{j+1})$.   
For a given $t$, since $\Sigma(t)$ is the covariance of $\U(t)$, with probability one, $\U(t) \in \range\Sigma(t)$, implying that $\|P(t) \U(t)\| = 0$. Let $\omega$ denote a generic element of the background sample space. We just found that, for all $t$, $h(t, \omega) := \|P(t) \U(t, \omega)\| = 0$ for almost all $\omega$. By Fubini's theorem, this implies that for almost all $\omega$, $h(t, \omega) = 0$ for almost all $t$. By the fact that $\U$ is a continuous process, for almost all $\omega$, $h(t, \omega) = 0$ for almost all $t$ and $t \mapsto U(t, \omega)$ is continuous. For such an $\omega$, $t \mapsto h(t, \omega)$ is continuous on each interval $(t_j, t_{j+1})$, and being zero almost everywhere, it must be zero everywhere on each such interval. We thus found that, with probability one, $\U(t) \in \range\Sigma(t)$ for all $t$ except possibly $\{t_1, \dots, t_m\}$. However, we know that for each given $t$, with probability one, $\U(t) \in \range\Sigma(t)$. We can thus conclude that with probability one, $\U(t) \in \range\Sigma(t)$ for all $t$.
\end{proof}

\end{document}